\documentclass[a4paper, reqno, 10pt, headings=standardclasses, twoside, numbers=endperiod, DIV=8, BCOR=0cm]{scrartcl}

\usepackage{amsmath, amsthm, amssymb}
\usepackage{mathtools}
\usepackage{xcolor}
\usepackage{fixltx2e}
\usepackage{enumitem}
\usepackage{url}
\usepackage{lmodern}
\usepackage{eucal}
\usepackage[hyperfootnotes=false]{hyperref}
\usepackage[left=3cm, right=3cm, top=3cm, bottom=3cm]{geometry}
\usepackage[capitalise, noabbrev, nameinlink]{cleveref}
\usepackage{scrlayer-scrpage}

\clearscrheadfoot
\cehead{B. Topacogullari}
\cofoot*{\pagemark}
\cohead{The fourth moment of Dirichlet \(L\)-functions}
\cefoot*{\pagemark}

\addtokomafont{section}{\large}
\addtokomafont{subsection}{\normalsize}

\RedeclareSectionCommand[afterskip=-0.75em]{subsection}

\definecolor{color_link}{rgb}{0.0,0.1,0.85}
\hypersetup{hypertexnames=false, colorlinks=true, citecolor=color_link, linkcolor=color_link, urlcolor=color_link, pdfauthor=Berke Topacogullari, pdftitle=The fourth moment of individual Dirichlet L-functions on the critical line}

\makeatletter
  \def\blfootnote{\gdef\@thefnmark{}\@footnotetext}
\makeatother

\theoremstyle{plain}
\newtheorem{theorem}{Theorem}[section]
\newtheorem{lemma}[theorem]{Lemma}
\newtheorem{proposition}[theorem]{Proposition}
\numberwithin{equation}{section}

\let\originalleft\left
\let\originalright\right
\renewcommand{\left}{\mathopen{}\mathclose\bgroup\originalleft}
\renewcommand{\right}{\aftergroup\egroup\originalright}

\renewcommand\Re{\operatorname{Re}}
\renewcommand\Im{\operatorname{Im}}
\newcommand\GL{\mathrm{GL}}
\newcommand\twoln[2]{{\substack{#1 \\ #2}}}
\newcommand\thrln[3]{{\substack{#1 \\ #2 \\ #3}}}

\newcommand{\arcosh}{\operatorname{arcosh}}

\renewcommand{\d}{\mathrm{d}}
\renewcommand{\i}{\mathrm{i}}
\newcommand\Res[1]{\underset{#1}{\operatorname{Res}}\,\,}
\newcommand\mf[1]{\mathfrak{#1}}

\renewcommand\O{O}
\newcommand\LO[1]{\O\left(#1\right)}
\newcommand\supp{\operatorname{supp}}

\newcommand\mc[1]{\mathcal{#1}}

\newcommand\eps\varepsilon
\renewcommand\phi\varphi

\newcommand\ZZ{\mathbb{Z}}
\newcommand\QQ{\mathbb{Q}}
\newcommand\RR{\mathbb{R}}
\newcommand\CC{\mathbb{C}}

\title{\LARGE The fourth moment of individual Dirichlet \(L\)-functions on the critical line}
\author{\Large Berke Topacogullari}

\date{%
  \vskip\baselineskip
  \normalfont\normalsize%
  \parbox{0.8\linewidth}{%
    \noindent
    \textbf{Abstract.} We prove an asymptotic formula for the second moment of a product of two Dirichlet \(L\)-functions on the critical line, which has a power saving in the error term and which is uniform with respect to the involved Dirichlet characters.
    As special cases we give uniform asymptotic formulae for the fourth moment of individual Dirichlet \(L\)-functions and for the second moment of Dedekind zeta functions of quadratic number fields on the critical line.
  }
  \vskip\baselineskip
}

\begin{document}

\maketitle

\blfootnote{\textup{2010} \textit{Mathematics Subject Classification:} 11M06}

\section{Introduction} \label{1}

Moments of \(L\)-functions are a central topic in analytic number theory, not only due to their many important applications, but also because they give insight into the behaviour of \(L\)-functions in the critical strip.

One of the most famous and best-studied examples in this regard is the fourth moment of the Riemann zeta function
\begin{equation} \label{1063}
  \int_1^T \! \left| \zeta\left( \tfrac12 + \i t \right) \right|^4 \, \d t.
\end{equation}
The first asymptotic formula for~\eqref{1063} goes back to Ingham~\cite{Ing27}, who proved that
\[ \int_1^T \! \left| \zeta\left( \tfrac12 + \i t \right) \right|^4 \, \d t = \frac1{2 \pi^2} T (\log T)^4 + \LO{ T (\log T)^3 }. \]
It was not until several decades later that Heath-Brown~\cite{Hea79} was able to improve on this estimate.
His result, which marked a major advance in the subject, states that
\begin{equation} \label{1054}
  \int_1^T \! \left| \zeta\left( \tfrac12 + \i t \right) \right|^4 \, \d t = T P(\log T) + \LO{ T^{\frac78 + \eps} },
\end{equation}
where \(P\) is a certain polynomial of degree~\(4\).
Further progress came with the development of methods originating in the spectral theory of automorphic forms, in particular the Kuznetsov formula~\cite{Kuz80}.
Zavorotnyi~\cite{Zav89} was thus able to lower the exponent in the error term in~\eqref{1054} and show that
\begin{equation} \label{1099}
  \int_1^T \! \left| \zeta\left( \tfrac12 + \i t \right) \right|^4 \, \d t = T P(\log T) + \LO{ T^{\frac23 + \eps} }.
\end{equation}
Motohashi~\cite[Theorem~4.2]{Mot97a} established an explicit formula which expresses a smooth version of the fourth moment~\eqref{1063} in terms of the cubes of the central values of certain automorphic \(L\)-functions.
His result is significant, as it allows a much deeper understanding of~\eqref{1063} than a mere asymptotic estimate, in addition to having many remarkable applications (see e.g.~\cite{IM94, IM95}).
The best estimate for~\eqref{1063} to date is due to Ivi{\'c} and Motohashi~\cite[Theorem~1]{IM95} who, by making use of the explicit formula, were able to replace the factor~\( T^\eps \) in~\eqref{1099} by a suitable power of~\( \log T \).

In this article, we are interested in the analogous problem for Dirichlet \(L\)-functions.
Naturally, the fourth moment can here be taken in two different ways:
On the one hand, we can look at an individual Dirichlet \(L\)-function and take the average along the critical line as in~\eqref{1063}.
On the other hand, we can focus on the central point~\( s = 1/2 \) and take the average over a suitable subset of Dirichlet characters, most typically the set of all primitive Dirichlet characters of a given modulus~\(q\).

The latter case has probably received most of the attention.
The first result goes back to Heath-Brown~\cite{Hea81}, who proved an asymptotic formula for those~\(q\) with not too many prime factors, which was later extended by Soundararajan~\cite{Sou07} to all~\(q\).
Young~\cite{You11} achieved a major breakthrough when he proved, for \(q\)~prime, an asymptotic formula with a power saving in the error term.
His result states that
\begin{equation} \label{1030}
  \sideset{}{^\ast} \sum_{ \chi \bmod q } \left| L\left( \tfrac12, \chi \right) \right|^4 = \phi(q)^\ast P(\log q) + \LO{ q^{ 1 - \frac5{512} + \eps } },
\end{equation}
where the~\( \ast \) on the sum indicates that the sum is restricted to primitive Dirichlet characters, where \( \phi^\ast(q) \) denotes the number of primitive characters mod~\(q\), and where \( P \) is a certain polynomial of degree~\(4\).
As in the works of Zavorotnyi~\cite{Zav89} and Motohashi~\cite{Mot97a}, his proof relies crucially on methods coming from the spectral theory of automorphc forms.
The exponent~\( 5/512 \) in the error term was later improved to~\( 1/20 \) by Blomer, Fouvry, Kowalski, Michel and Mili{\'c}evi{\'c}~\cite{BFKMM17a, BFKMM17b}.

A few results are also available if an additional average over~\(t\) is included.
Rane~\cite{Ran81} showed that
\begin{equation} \label{1060}
  \sideset{}{^\ast} \sum_{\chi \bmod q} \int_T^{2T} \! \left| L\left( \tfrac12 + \i t \right) \right|^4 \, \d t = C(q) \phi^\ast(q) T ( \log qT )^4 + \LO{ 2^{ \omega(q) } \phi^\ast(q) T (\log qT)^3 (\log\log 3q)^5 },
\end{equation}
where \( \omega(q) \) denotes the number of prime factors of~\(q\), and where \( C(q) \) is a certain constant depending on~\(q\).
This is an asymptotic formula in certain ranges of \(q\)~and~\(T\).
Bui and Heath-Brown~\cite{BH10} sharpened the error term in~\eqref{1060}, and established an asymptotic formula when~\( q \) goes to infinity.
Another result is due to Wang~\cite{Wan91}, who proved that, for~\( q \leq T \),
\begin{equation} \label{1052}
  \sideset{}{^\ast} \sum_{\chi \bmod q} \int_0^T \! \left| L\left( \tfrac12 + \i t \right) \right|^4 \, \d t = \phi(q)^\ast T P_q\left( \log T \right) + \LO{ \min\left\{ q^\frac98 T^{ \frac78 + \eps }, q T^{ \frac{11}{12} + \eps } \right\} },
\end{equation}
where \( P_q \) is a certain polynomial of degree~\(4\) with coefficients depending on~\(q\).

The direct analogue of~\eqref{1063}, that is the fourth moment of an individual Dirichlet \(L\)-function on the critical line
\begin{equation} \label{1019}
  \int_1^T \left| L\right( \tfrac12 + \i t, \chi \left) \right|^4 \d t,
\end{equation}
has received much less attention.
If \(\chi\) is considered fixed, then a simple asymptotic formula for~\eqref{1019} can be obtained by classical methods, although this has not been worked out explicitly in the literature.
It is a much more difficult problem to obtain estimates uniform in~\(\chi\) and comparable in strength to what can be achieved for~\( \zeta(s) \).
It is this latter problem which we want to address here.

Our main result is as follows.

\begin{theorem} \label{101}
  Let~\( \eps > 0 \).
  Let \( \chi \)~mod~\(q\) be a primitive Dirichlet character.
  Then we have, for~\( T \geq 1 \),
  \begin{equation} \label{1096}
    \int_1^T \left| L\left( \tfrac12 + \i t, \chi \right) \right|^4 \d t = \int_1^T \! P_\chi (\log t) \, \d t + \LO{ q^{2 - 3\theta} T^{\frac12 + \theta + \eps} + q T^{\frac23 + \eps} },
  \end{equation}
  where \( P_\chi \) is a polynomial of degree~\(4\), whose coefficients depend only on~\(q\), and where the implicit constant depends only on~\(\eps\).
\end{theorem}

Here \(\theta\) denotes the bound in the Ramanujan-Petersson conjecture (see \cref{31} for a precise definition).
By the work of Kim and Sarnak~\cite{Kim03} it is known that \( \theta = 7/64 \)~is admissible, and with this value our asymptotic formula is non-trivial in the range~\( q \ll T^{ 25 / 107 - \eps } \).
The polynomial~\( P_\chi \) appearing in the main term can be described fairly explicitly in form of a residue (see~\eqref{5508}).
In particular, its leading coefficient is given by
\[ \frac1{2\pi^2} \frac{ \phi(q)^2 }{q^2} \prod_{p \mid q} \left( 1 - \frac2{p + 1} \right). \]
This constant also appears as leading coefficient in the polynomials in~\eqref{1030} and~\eqref{1052}, and is identical to the constant~\( C(q) \) in~\eqref{1060}.
With a couple of minor technical modifications in the proof, \cref{101} can be extended to all Dirichlet characters.

A similar formula holds if we replace the sharp integration bounds in~\eqref{1019} by a smooth weight function.

\begin{theorem} \label{102}
  Let~\( \eps > 0 \).
  Let~\( w : (0, \infty) \to \CC \) be a smooth and compactly supported function.
  Let \( \chi \)~mod~\(q\) be a primitive Dirichlet character.
  Then we have, for~\( T \geq 1 \),
  \[ \int \left| L\left( \tfrac12 + \i t, \chi \right) \right|^4 w\left( \frac tT \right) \d t = \int \! P_\chi (\log t) w\left( \frac tT \right) \, \d t + \LO{ q^{2 - 3\theta} T^{\frac12 + \theta + \eps} }, \]
  where \( P_\chi \) is the same polynomial as in~\eqref{1096}, and where the implicit constant depends only on~\(w\) and~\(\eps\).
\end{theorem}

An interesting generalization of~\eqref{1019} concerns the mixed moment
\begin{equation} \label{1045}
  \int_1^T \left| L\left( \tfrac12 + \i t, \chi_1 \right) \right|^2 \left| L\left( \tfrac12 + \i t, \chi_2 \right) \right|^2 \d t,
\end{equation}
where \( \chi_1 \)~and~\( \chi_2 \) are two different primitive Dirichlet characters.
In general, it is expected that the behaviour of the two Dirichlet \(L\)-functions~\( L(s, \chi_1) \) and~\( L(s, \chi_2) \) on the critical line is uncorrelated, which should also find its expression in a slightly different asymptotic behaviour of~\eqref{1045} compared with~\eqref{1019}.
Specifically, heuristical considerations suggest that the mixed moment~\eqref{1045} should have a leading term of the order of~\( T (\log T)^2 \) instead of~\( T (\log T)^4 \) (see~\cite{MT14} for a discussion of this phenomenon in a more general context).

This is indeed the case as our next result confirms.

\begin{theorem} \label{103}
  Let~\( \eps > 0 \).
  Let \( \chi_1 \)~mod~\(q_1\) and \( \chi_2 \)~mod~\(q_2\) be two different primitive Dirichlet characters, and let
  \begin{equation} \label{1092}
    q_1^\star := \left( q_1, {q_2}^\infty \right) / (q_1, q_2) \qquad \text{and} \qquad q_2^\star := \left( q_2, {q_1}^\infty \right) / (q_1, q_2).
  \end{equation}
  Then we have, for~\( T \geq 1 \),
  \begin{align}
    \begin{split} \label{1022}
      \int_1^T &\left| L\left( \tfrac12 + \i t, \chi_1 \right) \right|^2 \left| L\left( \tfrac12 + \i t, \chi_2 \right) \right|^2 \d t = \int_1^T \! P_{\chi_1, \chi_2}(\log t) \, \d t \\
        &\qquad\qquad + \LO{ ( q_1^\star q_1 + q_2^\star q_2 )^\frac12 (q_1 q_2)^{ \frac34 - \frac32\theta } T^{ \frac12 + \theta + \eps } + ( q_1^\star q_1 + q_2^\star q_2 )^\frac13 (q_1 q_2)^\frac13 T^{ \frac23 + \eps } },
    \end{split}
  \end{align}
  where \( P_{\chi_1, \chi_2} \) is a quadratic polynomial, whose coefficients depend only on~\( \chi_1 \) and~\(\chi_2\), and where the implicit constant depends only on~\(\eps\).
\end{theorem}

As before, the polynomial~\( P_{\chi_1, \chi_2} \) appearing in the main term can be stated explicitly (see~\eqref{5583} and~\eqref{5576}).
Its leading coefficient is given by
\[ \frac6{\pi^2} \left| L( 1, \overline{\chi_1} \chi_2 ) \right|^2 \frac{ \phi(q_1) \phi(q_2) }{ \phi(q_1 q_2) } \prod_{ p \mid q_1 q_2 } \left( 1 - \frac1{p + 1} \right). \]
On a side note, this result also shows that for a given primitive, non-real Dirichlet character~\(\chi\) there is no correlation between the functions~\( L(1/2 + \i t, \chi) \) and~\( L(1/2 - \i t, \chi) \).

The analogue of \cref{103} for the smooth moment reads as follows.

\begin{theorem} \label{104}
  Let~\( \eps > 0 \).
  Let~\( w : (0, \infty) \to \CC \) be a smooth and compactly supported function.
  Let \( \chi_1 \)~mod~\(q_1\) and \( \chi_2 \)~mod~\(q_2\) be two different primitive Dirichlet characters, and let \( q_1^\star \) and~\( q_2^\star \) be defined as in~\eqref{1092}.
  Then we have, for~\( T \geq 1 \),
  \begin{multline*}
    \int \left| L\left( \tfrac12 + \i t, \chi_1 \right) \right|^2 \left| L\left( \tfrac12 + \i t, \chi_2 \right) \right|^2 w\left( \frac tT \right) \, \d t = \int \! P_{\chi_1, \chi_2}(\log t) w\left( \frac tT \right) \, \d t \\
      + \LO{ ( q_1^\star q_1 + q_2^\star q_2 )^\frac12 (q_1 q_2)^{ \frac34 - \frac32\theta } T^{ \frac12 + \theta + \eps } },
  \end{multline*}
  where \( P_{\chi_1, \chi_2} \) is the same polynomial as in~\eqref{1022}, and where the implicit constant depends only on~\(w\) and~\(\eps\).
\end{theorem}

A certain special case of \cref{103} deserves its own mention.
If~\(K\) is a quadratic number field with discriminant~\(D\), then it is well-known that the Dedekind zeta function~\( \zeta_K(s) \) associated to~\(K\) has the form
\[ \zeta_K(s) = \zeta(s) L(s, \chi_D), \]
where \(\chi_D\) is a certain real primitive Dirichlet character of modulus~\( |D| \).
Hence, by applying \cref{103} on this product of Dirichlet \(L\)-functions, we get the following asymptotic formula for the second moment of~\( \zeta_K \) on the critical line.

\begin{theorem} \label{105}
  Let~\( \eps > 0 \).
  Let \(K\) be a quadratic number field with discriminant~\(D\).
  Then we have, for~\( T \geq 1 \),
  \begin{equation} \label{1034}
    \int_1^T \left| \zeta_K\left( \tfrac12 + \i t \right) \right|^2 \d t = \int_1^T \! P_K(\log t) \, \d t + \LO{ |D|^\frac23 T^{\frac23 + \eps} + |D|^{ \frac54 -\frac 32 \theta } T^{ \frac12 + \theta + \eps } },
  \end{equation}
  where \( P_K \) is a quadratic polynomial, whose coefficients depend only on the field~\( K \), and where the implicit constant depends only on~\(\eps\).
\end{theorem}

This improves on previous results by Motohashi~\cite{Mot70}, Hinz~\cite{Hin79} and M{\"u}ller~\cite{Mue89}.
With the current best value for~\(\theta\), the asymptotic formula is non-trivial as long as~\( |D| \ll T^{ 50/139 - \eps } \).
The leading constant of~\( P_K \) is
\[ \frac6{\pi^2} \left| L( 1, \chi_D ) \right|^2 \prod_{p \mid D} \left( 1 - \frac1{p + 1 } \right). \]

We also want to formulate the analogue of \cref{105} for the smooth moment.

\begin{theorem} \label{106}
  Let~\( \eps > 0 \).
  Let~\( w : (0, \infty) \to \CC \) be a smooth and compactly supported function.
  Let \(K\) be a quadratic number field with discriminant~\(D\).
  Then we have, for~\( T \geq 1 \),
  \[ \int \left| \zeta_K\left( \tfrac12 + \i t \right) \right|^2 w\left( \frac tT \right) \, \d t = \int \! P_K(\log t) w\left( \frac tT \right) \, \d t + \LO{ |D|^{ \frac54 -\frac 32 \theta } T^{ \frac12 + \theta + \eps } }, \]
  where \( P_K \) is the same polynomial as in~\eqref{1034}, and where the implicit constant depends only on~\(w\) and~\(\eps\).
\end{theorem}

We did not attempt to establish explicit formulae of the type Motohashi established for~\( \zeta(s) \), as this would have further complicated many of the already complicated estimations done in the proof.
Nevertheless, it would certainly be interesting to develop such identites for the moments considered here, in particular for the fourth moment of Dirichlet \(L\)-functions.
In fact, for the second moment of Dedekind zeta functions of quadratic number fields, an explicit formula has been worked out by Motohashi~\cite{Mot93, Mot97b} (see also~\cite{BHKM19, BM01, BM03} for other related results).

We now proceed to give an overview of the proof of our results, focusing here on \cref{101}.
For the most part, we follow rather classical paths, taken in similar forms in many of the works cited above.
By the use of a suitable approximate functional equation for the square~\( L(s, \chi)^2 \), we express the quantity~\( \left| L\left( 1/2 + \i t, \chi \right) \right|^4 \) as a finite double Dirichlet series of roughly the form
\[ \sum_{ n_1, n_2 \ll q T } \!\!\! \frac{ \chi(n_1) \overline\chi(n_2) \tau(n_1) \tau(n_2) }{ (n_1 n_2)^\frac12 } \left( \frac{n_2}{n_1} \right)^{\i t} + \alpha_\chi\left( \tfrac12 + \i t \right) \!\! \sum_{ n_1, n_2 \ll q T } \!\! \frac{ \chi(n_1) \overline\chi(n_2) \tau(n_1) \tau(n_2) }{ (n_1 n_2)^\frac12 } (n_1 n_2)^{\i t}, \]
where \( \tau(n) \) denotes the usual divisor function and where \( \alpha_\chi(s) \) is given by
\[ \alpha_\chi(s) := L(s, \chi)^2 / L(1 - s, \overline\chi)^2. \]
Once this is established, we simply integrate term-wise over~\(t\).
This operation has a localizing effect on the sum on the left, in the sense that only those terms remain where \(n_1\)~and~\(n_2\) are not too far apart, all other terms becoming negligibly small due to the oscillation in~\(t\).
The sum on the right effectively disappears as a whole because of oscillatory effects coming from the two factors~\( \alpha_\chi(1/2 + \i t) \) and~\( (n_1 n_2)^{\i t} \).

Eventually, two different sums remain which we need to estimate.
On the one hand, we have the contribution coming from the diagonal terms~\( n_1 = n_2 \), which takes the shape
\begin{equation} \label{1095}
  \sum_\twoln{ n \ll q T }{ (n, q) = 1 } \frac{ \tau(n)^2  }n,
\end{equation}
and which can be evaluated rather easily, giving rise to the first two leading terms in the final asymptotic formula~\eqref{1096}.
On the other hand, we have the contribution coming from the off-diagonal terms, which -- ignoring here any remaining oscillatory factors -- roughly look as follows,
\[ \sum_\twoln{ n_1, n_2 \ll q T }{ 0 < | n_1 - n_2 | \ll T^{1/3} } \frac{ \chi(n_1) \overline\chi(n_2) \tau(n_1) \tau(n_2) }{ (n_1 n_2)^\frac12 \log( n_2 / n_1 ) }. \]
It also contributes to the main term in the end, although only to the lower order terms.
It is, however, considerably harder to analyze than~\eqref{1095}, and its evaluation forms the actual core of the proof of \cref{101}.

After reordering the terms according to the value of~\( h := n_2 - n_1 \), we arrive at the following type of sums,
\begin{equation} \label{1002}
  \sum_{ n \ll qT } \chi(n) \overline\chi(n + h) \tau(n) \tau(n + h),
\end{equation}
where the parameter~\(h\) can be as large as~\( T^{1/3} \).
This is an instance of the so-called shifted convolution problem, which comes up regularly in the study of the analytic behaviour of \(L\)-functions.
Similar sums also appeared for instance in the works of Heath-Brown~\cite{Hea79} and Young~\cite{You11} cited above.
In our case, it is the presence of the Dirichlet characters which complicates the analysis considerably, leading to several technical difficulties down the road, in particular with regard to the application of spectral methods.

The crucial point in the evaluation of~\eqref{1002} comes after a couple of initial transformations, when we encounter sums of Kloosterman sums of roughly the following form,
\begin{equation} \label{1070}
  \sum_{a \bmod q} \chi(a) \overline{\chi}(m - a) \sum_{ (c, q) = 1 } S( \overline c^2 h, a; q ) S( \overline q^2 h, m; c ) F(c),
\end{equation}
where \(m\) is an integer and where \( F(c) \) is some weight function.
Ideally, at this point one would like to estimate the sum of Kloosterman sums over~\(c\) via the Kuznetsov formula, while also exploiting the cancellation in the character sum over~\(a\).
However, already the first task brings serious difficulties, as it is not clear in which form -- if there is any -- the Kuznetsov formula might be applicable here.

The route we take to solve this problem is to write the first Kloosterman sum in terms of Dirichlet characters as follows (assuming for simplicity that \(h\)~and~\(q\) are coprime),
\begin{equation} \label{1078}
  S( \overline c^2 h, a; q ) = \frac1{ \phi(q) } \sum_{\psi \bmod q} \psi(c)^2 \overline\psi(ha) G(\psi)^2,
\end{equation}
where the sum runs over all Dirichlet characters mod~\(q\), and where \( G(\psi) \) denotes the Gau{\ss} sum associated to~\(\psi\).
The idea underlying this approach goes initially back to Blomer and Mili{\'c}evi{\'c}~\cite{BM15}, and was used in similar forms also in other works (see~\cite{PY19, Top17, Zac16}).
It allows us to separate the two variables \(a\)~and~\(c\) in~\eqref{1070}, while at the same time bringing the sum of Kloosterman sums into a form susceptible to the use of the Kuznetsov formula.

Of course taking this route comes with a cost:
Encoding the Kloosterman sum~\( S( \overline c^2 h, a; q ) \) via Dirichlet characters introduces an additional factor of the size of~\( q^{1/2} \), which we cannot get rid of afterwards and which inevitably turns up in the error term in \cref{101}.

We suspect that there should be a more direct way to employ the Kuznetsov formula on the sum~\eqref{1070}, which avoids the rather artificial detour via~\eqref{1078} taken here.
This might not only lead to an improvement of the error term in \cref{101} in the \(q\)-aspect, but would also prove extremely useful when trying to establish an explicit formula of Motohashi type for the fourth moment of Dirichlet \(L\)-functions (see also the comments in~\cite[pp.~182--183]{Mot97a} on this matter).

\subsection*{Plan}

The article is organized as follows.
In \cref{2}, we introduce the basic notation used throughout the article, and state some technical results related to Dirichlet \(L\)-functions.
In \cref{3}, we briefly present the needed tools from the spectral theory of automorphic forms.
In \cref{4}, we consider the shifted convolution problem lying at the heart of the proof of our results.
Finally, in \cref{5}, we proof Theorems~\ref{101}--\ref{106}.
The last two sections can be read independently of each other.

\subsection*{Acknowledgements}

I would like to thank V.~Blomer, J.~B. Conrey, Y.~Motohashi, R.~M. Nunes and M.~P. Young for valuable discussions and remarks.
In particular, I am grateful to J.~B. Conrey for making me aware of his article~\cite{Con96}, which was very helpful in the evaluation of the main terms in \cref{55}.

\section{Background on Dirichlet \texorpdfstring{\(L\)}{L}-functions} \label{2}

The aim of section is to introduce the basic notation used in the following, and state a couple of technical lemmas related to Dirichlet \(L\)-functions.

\subsection{Notation} \label{21}

We will use the convention that~\(\eps\) denotes a positive real number which can be chosen arbitrarily small and whose value may change at each occurrence.
We write~\( A \asymp B \) to mean~\( A \ll B \ll A \).

We denote the Gau{\ss} sum associated to the Dirichlet character~\( \chi \)~mod~\(q\) by
\[ G(\chi, h) := \sum_{a \bmod q} \chi(a) e\left( \frac{ah}q \right), \]
where as usual~\( e(\xi) := \exp( 2 \pi \i \xi ) \).
We set~\( G(\chi) := G(\chi, 1) \).
Other frequently occurring exponential sums are the Ramanujan sums and Kloosterman sums, for which we will use the notations
\[ r_q(h) := \sum_\twoln{ a \bmod q }{ (a, q) = 1 } e\left( \frac{ah}q \right) = \sum_{ d \mid (h, q) } \mu\left( \frac qd \right) d \qquad \text{and} \qquad S(m, n; q) := \sum_\twoln{a \bmod q}{ (a, q) = 1 } e\left( \frac{ m a + n \overline a }q \right), \]
where \( \overline a \)~indicates a solution to~\( \overline a a \equiv 1 \bmod q \).

Let \( \chi_1 \)~mod~\(q_1\) and \( \chi_2 \)~mod~\(q_2\) be Dirichlet characters, which throughout the article will be assumed to be primitive.
We denote the product of the two Dirichlet \(L\)-functions~\( L(s, \chi_1) \) and~\( L(s, \chi_2) \) by~\( L_{\chi_1, \chi_2}(s) := L(s, \chi_1) L(s, \chi_2) \).
For~\( \Re(s) > 1 \), this function can be written as a Dirichlet series,
\[ L_{\chi_1, \chi_2}(s) = \sum_{n = 1}^\infty \frac{ \tau_{\chi_1, \chi_2}(n) }{n^s} \qquad \text{with} \qquad \tau_{\chi_1, \chi_2}(n) := \sum_{d \mid n} \chi_1(d) \chi_2\left( \frac nd \right). \]
Furthermore, it satisfies the following functional equation,
\[ L_{\chi_1, \chi_2}(s) = \alpha_{\chi_1, \chi_2}(s) L_{ \overline{\chi_1}, \overline{\chi_2} }(1 - s), \]
with \( \alpha_{\chi_1, \chi_2}(s) \) given by
\[ \alpha_{\chi_1, \chi_2}(s) := \frac{ G(\chi_1) G(\chi_2) }{ \pi^2 \i^{ \kappa(\chi_1) + \kappa(\chi_2) } } \left( \frac{4 \pi^2}{q_1 q_2} \right)^s \sin\left( \frac\pi2 \left( s + \kappa(\chi_1) \right) \right) \sin\left( \frac\pi2 \left( s + \kappa(\chi_2) \right) \right) \Gamma(1 - s)^2, \]
where we have set
\begin{equation} \label{2109}
  \kappa(\chi_i) := ( 1 - \chi_i(-1) ) / 2.
\end{equation}

\subsection{Estimates for~\texorpdfstring{\( \alpha_{\chi_1, \chi_2}(s) \)}{alpha(s)} and~\texorpdfstring{\( L_{\chi_1, \chi_2}(s) \)}{L(s)}} \label{22}

We will need rather precise estimates for~\( \alpha_{\chi_1, \chi_2}(s) \) on the critical line.
By using a suitable approximation for the gamma function (see e.g.~\cite[Chapter~5,~(38)]{Ahl79}) we can write this quantity, for~\( |t| \geq 1 \), as
\begin{equation} \label{2263}
  \alpha_{\chi_1, \chi_2}\left( \tfrac12 + \i t \right) = \i \frac{ G(\chi_1) G(\chi_2) }{ (-1)^{ \kappa(\chi_1) + \kappa(\chi_2) } \sqrt{q_1 q_2} } e\left( \frac t\pi \log\left( \frac{2 \pi e}{ t \sqrt{q_1 q_2} } \right) \right) A(t),
\end{equation}
where \( A : \RR \to \CC \)~is a certain smooth function whose derivatives are bounded by
\[ A^{ (\nu) }(t) \ll |t|^{-\nu} \quad \text{for} \quad \nu \geq 0. \]
Note that we also have
\begin{equation} \label{2237}
  \left| \alpha_{\chi_1, \chi_2}\left( \tfrac12 + \i t\right) \right| = 1 \quad \text{for} \quad t \in \RR.
\end{equation}
In the critical strip, the following simple estimate will suffice,
\begin{equation} \label{2284}
  | \alpha_{\chi_1, \chi_2}(\sigma + \i t) | \asymp t^{1 - 2\sigma} (q_1 q_2)^{ \frac12 - \sigma } \qquad \text{for} \qquad \sigma \in [0, 1], \quad |t| \geq 1.
\end{equation}

Concerning~\( L_{\chi_1, \chi_2}(s) \), we have the following hybrid upper bound, which is an immediate consequence of a result by Heath-Brown~\cite{Hea80} and the convexity principle.

\begin{theorem} \label{221}
  Let~\( \eps > 0 \).
  We have, for~\( \sigma \in [0, 1] \) and~\( t \in \RR \) with~\( | \sigma + \i t - 1 | > \eps \),
  \[ L_{\chi_1, \chi_2}(\sigma + it) \ll ( q_1 q_2 )^{ \frac{ 3 (1 - \sigma) }8 + \eps }  ( |t| + 1 )^{ \frac{ 3 (1 - \sigma) }4 + \eps }, \]
  where the implicit constant depends only on~\(\eps\).
\end{theorem}

We will also need upper bounds for the first moment of~\( L_{\chi_1, \chi_2}(s) \) in the critical strip.
In this regard, the following result will be helpful.

\begin{theorem} \label{222}
  Let~\( \eps > 0 \).
  We have, for~\( \sigma \in [0, 1] \) and~\( q_1, q_2 \leq T \),
  \[ \int_1^T \left| L_{\chi_1, \chi_2}(\sigma + \i t) \right| \, \d t \ll T^{1 + \eps} + (q_1 q_2)^{\frac12 - \sigma} T^{2 - 2 \sigma + \eps}, \]
  where the implicit constant depends only on~\(\eps\).
\end{theorem}
\begin{proof}
  For~\( \sigma = 1/2 \), this is an immediate consequence of a result by Gallagher~\cite[(1\(_T\))]{Gal75}.
  His proof can easily be adapted to cover also the range~\( \sigma > 1/2 \), and the result for~\( \sigma < 1/2 \) then follows from the functional equation and~\eqref{2284}.
\end{proof}

\subsection{Voronoi summation for~\texorpdfstring{\( \tau_{\chi_1, \chi_2}(n) \)}{tau(n)}} \label{23}

Here we want to develop a summation formula of Voronoi type for~\( \tau_{\chi_1, \chi_2}(n) \).

Before stating the result we need to introduce some notation.
Let~\(a\) and~\( c > 0 \) be coprime integers.
We set
\[ \hat \tau_{\chi_1, \chi_2}\left( n; \frac ac \right) := \frac1{ [c, q_1]^\frac12 [c, q_2]^\frac12 } \sum_{n_1 n_2 = n} \sum_\twoln{ b_1 \bmod [c, q_1] }{ b_2 \bmod [c, q_2] } \chi_1(b_1) \chi_2(b_2) e\left( \frac{ a b_1 b_2 }c + \frac{n_1 b_1}{ [c, q_1] } + \frac{n_2 b_2}{ [c, q_2] } \right), \]
where \( [c, q_i] \) denotes the least common multiple of \(c\)~and~\(q_i\).
We also define
\begin{align}
  B_{\chi_1, \chi_2}^+(\xi) &:= \left\{ \begin{aligned}
      -2\pi Y_0( 4\pi \xi ) \quad \text{if} \quad \chi_1(-1) = \chi_2(-1), \\
      -2\pi\i J_0( 4\pi \xi ) \quad \text{if} \quad \chi_1(-1) \neq \chi_2(-1),
    \end{aligned} \right. \label{2303} \\
  B_{\chi_1, \chi_2}^-(\xi) &:= 2 \left( \chi_1(-1) + \chi_2(-1) \right) K_0( 4\pi \xi ). \label{2352}
\end{align}
Finally, we define~\( \Pi_{\chi_1, \chi_2}( X; c, a ) \) to be the polynomial in~\(X\), which in the case~\( \chi_1 = \chi_2 \) is given by
\begin{align*}
  \Pi_{\chi_1, \chi_2}( X; c, a ) &:= \chi_1\left( \frac c{ (c, q_1) } \right) \overline{\chi_1}\left( \frac{ a q_1 }{ (c, q_1) } \right) G(\chi_1) \sum_{d \mid q_1} \frac{ \mu(d) }d \left( X + 2 \gamma + 2 \log\left( \frac{q_1}{cd} \right) \right), \\
  \intertext{and which otherwise is equal to the constant}
  \Pi_{\chi_1, \chi_2}( X; c, a ) &:= \chi_1\left( \frac c{ (c, q_2) } \right) \overline{\chi_2}\left( \frac{ q_2 }{ (c, q_2) } \right) G(\chi_2, a) L( 1, \chi_1 \overline{\chi_2} ) \\
    &\phantom{ := {} } \qquad\qquad\qquad + \chi_2\left( \frac c{ (c, q_1) } \right) \overline{\chi_1}\left( \frac{ q_1 }{ (c, q_1) } \right) G(\chi_1, a) L( 1, \overline{\chi_1} \chi_2 ).
\end{align*}

The Voronoi formula for~\( \tau_{\chi_1, \chi_2}(n) \) now reads as follows.

\begin{theorem} \label{231}
  Let \( f : (0, \infty) \to \CC \) be a smooth and compactly supported function.
  Let \(a\)~and~\( c \geq 1 \) be coprime integers.
  Let \( \chi_1 \)~mod~\(q_1\) and \(\chi_2\)~mod~\(q_2\) be primitive Dirichlet characters.
  Then
  \begin{align} \label{2367}
    \begin{split}
      \sum_n f(n) &\tau_{\chi_1, \chi_2}(n) e\left( \frac{an}c \right) = \frac1c \int \! \Pi_{\chi_1, \chi_2}( \log \xi; c, a ) f(\xi) \, \d\xi \\
        &\qquad + \frac1{ [c, q_1]^\frac12 [c, q_2]^\frac12 } \sum_\pm \sum_{n = 1}^\infty \hat \tau_{\chi_1, \chi_2}\left( n; \frac{\pm a}c \right) \int \! B_{\chi_1, \chi_2}^\pm\left( \frac{ (n \xi)^\frac12 }{ [c, q_1]^\frac12 [c, q_2]^\frac12 } \right) f(\xi) \, \d\xi.
    \end{split}
  \end{align}
\end{theorem}
\begin{proof}
  The proof of this result follows standard paths (see e.g.~\cite[Chapter~1]{Jut87}), although a few additional technical difficulties arise from the fact that the parameters \(c\), \(q_1\)~and~\(q_2\) may have possible common factors.
  To simplify the notation we set~\( c_1 := [c, q_1] \) and~\( c_2 := [c, q_2] \).
  
  We start by defining the following two Dirichlet series,
  \[ L_{\chi_1, \chi_2}\left( s; \frac ac \right) := \sum_{n = 1}^\infty \frac{ \tau_{\chi_1, \chi_2}(n) }{n^s} e\left( \frac{an}c \right) \qquad \text{and} \qquad \hat L_{\chi_1, \chi_2}\left( s; \frac ac \right) := \sum_{n = 1}^\infty \frac{ \hat \tau_{\chi_1, \chi_2}\left( n; \frac ac \right) }{n^s}. \]
  By expressing these two Dirichlet series in terms of Hurwitz zeta functions, we see that they can both be continued meromorphically to the whole complex plane with at most one possible pole at~\( s = 1 \) of degree not larger than~\(2\).
  In the same way, by using the functional equation for the Hurwitz zeta function, we deduce the following functional equation for~\( L_{\chi_1, \chi_2}(s; a/c) \),
  \begin{equation} \label{2332}
    L_{\chi_1, \chi_2}\left( s; \frac ac \right) = \frac{ \Gamma(1 - s)^2 }\pi \left( \frac{ 4 \pi^2 }{ c_1 c_2 } \right)^{s - \frac12} \sum_\pm \kappa_{\chi_1, \chi_2}^\pm(1 - s) \hat L_{\chi_1, \chi_2}\left( 1 - s; \pm \frac ac \right),
  \end{equation}
  with
  \[ \kappa_{\chi_1, \chi_2}^+(s) := \frac{ \chi_1 \chi_2(-1) e^{\pi\i s} + e^{-\pi\i s} }2 \qquad \text{and} \qquad \kappa_{\chi_1, \chi_2}^-(s) := \frac{ \chi_1(-1) + \chi_2(-1) }2. \]
  
  In order to prove~\cref{231}, we first express the sum on the left hand side in~\eqref{2367} via Mellin inversion as
  \[ \sum_n f(n) \tau_{\chi_1, \chi_2}(n) e\left( \frac{an}c \right) = \frac1{2\pi\i} \int_{ (2) } \! \hat f(s) L_{\chi_1, \chi_2}\left( s; \frac ac \right) \, \d s, \]
  where \(\hat f\) denotes the Mellin transform of~\(f\).
  After moving the line of integration to~\( \Re(s) = -1 \), using the functional equation~\eqref{2332}, and expanding the \(L\)-functions back into Dirichlet series, we arrive at
  \[ \sum_n f(n) \tau_{\chi_1, \chi_2}(n) e\left( \frac{an}c \right) = \Res{s = 1}\left( \hat f(s) L_{\chi_1, \chi_2}\left( s; \frac ac \right) \right) + \frac1{ (c_1 c_2)^\frac12 } \sum_\pm \sum_{n = 1}^\infty \hat \tau_{\chi_1, \chi_2}\left( n; \frac{\pm a}c \right) I^\pm(n), \]
  where
  \[ I^\pm(n) := \frac1{2\pi\i} \int_{ (-1) } \! G^\pm(1 - s) \hat f(s) \, \d s \qquad \text{with} \qquad G^\pm(s) := 2 \kappa_{\chi_1, \chi_2}^\pm(s) \Gamma(s)^2 \left( \frac{c_1 c_2}{4\pi^2 n} \right)^s. \]
  The integral~\( I^\pm(n) \) can be evaluated by observing that~\( G^\pm(s) \) is the Mellin transform of a certain Bessel function (see~\cite[17.43.16--18]{GR07}), so that by the Mellin convolution theorem we have
  \[ I^\pm(n) = \int \! B_{\chi_1, \chi_2}^\pm\bigg( \frac{ (n \xi)^\frac12 }{ (c_1 c_2)^\frac12 } \bigg) f(\xi) \, \d\xi, \]
  with \( B_{\chi_1, \chi_2}^\pm(\xi) \) as defined in \eqref{2303}~and~\eqref{2352}.
  
  It remains to evaluate the residue, which essentially amounts to determining the Laurent series expansion of~\( L_{\chi_1, \chi_2}( s; a/c) \) around~\( s = 1 \).
  We only want to indicate the main steps.
  Using the Laurent series expansion of the Hurwitz zeta function, we see that
  \[ L_{\chi_1, \chi_2}\left( s, \frac ac \right) = \frac1{ (s - 1)^2 } \lambda_{\chi_1, \chi_2}^{ (2) }\left( \frac ac \right) + \frac1{s - 1} \left( \lambda_{\chi_1, \chi_2}^{ (1) }\left( \frac ac \right) - \lambda_{\chi_1, \chi_2}^{ (2) }\left( \frac ac \right) \log(c_1 c_2) \right) + \LO{1}, \]
  where
  \begin{align*}
    \lambda_{\chi_1, \chi_2}^{ (2) }\left( \frac ac \right) &:= \frac1{c_1 c_2} \sum_\twoln{ 1 \leq b_1 \leq c_1 }{ 1 \leq b_2 \leq c_2 } \chi_1(b_1) \chi_2(b_2) e\left( \frac{a b_1 b_2}c \right), \\
    \lambda_{\chi_1, \chi_2}^{ (1) }\left( \frac ac \right) &:= -\frac1{c_1 c_2} \sum_\twoln{ 1 \leq b_1 \leq c_1 }{ 1 \leq b_2 \leq c_2 } \chi_1(b_1) \chi_2(b_2) e\left( \frac{a b_1 b_2}c \right) \left( \psi\left( \frac{b_1}{c_1} \right) + \psi\left( \frac{b_2}{c_2} \right) \right),
  \end{align*}
  with \(\psi\) denoting the digamma function.
  The first expression can be evaluated via~\cite[Lemma~5.4]{MV75}.
  For the second, we also make use of~\cite[Lemma~5.4]{MV75} and get
  \begin{multline*}
    \lambda_{\chi_1, \chi_2}^{ (1) }\left( \frac ac \right) = - \frac1{ c [q_1, q_2] } \Bigg( \chi_1\left( \frac{c_2}{q_2} \right) \overline{\chi_2}\left( a \frac{c_2}c \right) G(\chi_2) \sum_{b_1 = 1}^{ [q_1, q_2] } \chi_1 \overline{\chi_2}(b_1) \psi\left( \frac{b_1}{ [q_1, q_2] } \right) \\
      + \chi_2\left( \frac{c_1}{q_1} \right) \overline{\chi_1}\left( a \frac{c_1}c \right) G(\chi_1) \sum_{b_2 = 1}^{ [q_1, q_2] } \overline{\chi_1} \chi_2(b_2) \psi\left( \frac{b_2}{ [q_1, q_2] } \right) \Bigg).
  \end{multline*}
  The remaining sums can be calculated by writing \( L( s, \chi_1 \overline{\chi_2} ) \)~and~\( L( s, \overline{\chi_1} \chi_2 ) \) in terms of Hurwitz zeta functions and comparing the Laurent series coefficients around~\( s = 1 \).
\end{proof}

As an immediate corollary of \cref{231}, we can deduce a summation formula for~\( \tau_{\chi_1, \chi_2}(n) \) in arithmetic progressions.
If we set
\[ T_{\chi_1, \chi_2}(n; c, h) := \frac1{c^\frac12 } \sum_\twoln{a \bmod c}{ (a, c) = 1 } e\left( \frac{-h a}c \right) \hat \tau_{\chi_1, \chi_2}\left( n; \frac ac \right), \]
then the result reads as follows.

\begin{theorem} \label{232}
  Let \( f : (0, \infty) \to \CC \) be a smooth and compactly supported function.
  Let~\(h\) and~\( c \geq 1 \) be integers.
  Let \( \chi_1 \)~mod~\(q_1\) and \(\chi_2\)~mod~\(q_2\) be primitive Dirichlet characters.
  Then
  \begin{multline*}
    \sum_{ n \equiv h \bmod c } f(n) \tau_{\chi_1, \chi_2}(n) = \frac1c \sum_{c_0 \mid c} \frac1{c_0} \sum_\twoln{a_0 \bmod c_0}{ (a_0, c_0) = 1 } e\left( \frac{-h a_0}{c_0} \right) \int \! \Pi_{\chi_1, \chi_2}( \log\xi; c_0, a_0 ) f(\xi) \, \d\xi \\
      + \frac1c \sum_{c_0 \mid c} \frac{ {c_0}^\frac12 }{ [c_0, q_1]^\frac12 [c_0, q_2]^\frac12 } \sum_\pm \sum_{n = 1}^\infty T_{\chi_1, \chi_2}( n; c_0, \pm h ) \mc \int \! B_{\chi_1, \chi_2}^\pm\left( \frac{ (n \xi)^\frac12 }{ [c_0, q_1]^\frac12 [c_0, q_2]^\frac12 } \right) f(\xi) \, \d\xi.
  \end{multline*}
\end{theorem}
\begin{proof}
  The formula follows by encoding the congruence condition via additive characters and then applying \cref{231}.
\end{proof}

Concerning the Bessel function~\( B_{\chi_1, \chi_2}^+(\xi) \), we want to note the following technical lemma, which describes its behaviour for large~\(\xi\) (see~\cite[Lemma~2.3]{Top16}).

\begin{lemma} \label{233}
  If~\( \xi \gg 1 \), then \( B_{\chi_1, \chi_2}^+(\xi) \) can be expressed as
  \[ B_{\chi_1, \chi_2}^+(\xi) = 2 \Re\left( e( 2 \xi ) W_{\chi_1, \chi_2}(\xi) \right), \]
  where~\( W_{\chi_1, \chi_2} : (0, \infty) \to \CC \) is a certain smooth function whose derivatives satisfy the bounds
  \[ W_{\chi_1, \chi_2}^{ (\nu) }(\xi) \ll \xi^{-\frac12 - \nu} \quad \text{for} \quad \nu \geq 0. \]
\end{lemma}

We finish this section with the following result on Gau{\ss} sums, which is a special case of~\cite[Lemma~5.4]{MV75} and which will later be of use when evaluating the sums~\( T_{\chi_1, \chi_2}(n; c, h) \).

\begin{lemma} \label{234}
  Let \(\tilde \chi\)~mod~\(\tilde q\) be a Dirichlet character induced by the primitive character~\( \chi \)~mod~\( q \), and let \(a\) be an integer.
  Assume that~\( \tilde q \mid q^\infty \).
  Then \( G(\tilde \chi, a) \) vanishes unless \( \tilde q / q \) divides~\(a\), in which case we have
  \[ G(\tilde \chi, a) = \overline\chi\left( \frac{ a q }{\tilde q} \right) G(\chi) \frac{\tilde q}q. \]
\end{lemma}

\subsection{Approximate functional equations for~\texorpdfstring{\( L_{\chi_1, \chi_2}(s) \)}{L(s)}} \label{24}

Last but not least we want to state the following smooth approximate functional equation for~\( L_{\chi_1, \chi_2}(s) \) which generalizes~\cite[Theorem~4.2]{Ivi91} to Dirichlet \(L\)-functions.

\begin{theorem} \label{241}
  Let~\(\eps > 0 \).
  Let \( V : (0, \infty) \to [0, \infty) \) be a smooth function satisfying
  \[ V(\xi) + V( \xi^{-1} ) = 1 \quad \text{for} \quad \xi \in (0, \infty). \]
  Let~\( s = \sigma + \i t \in \CC \) and~\( x, y \geq 1 \) be such that~\( 1/2 \leq \sigma \leq 1 \), \( q_1, q_2 \leq t \) and~\( 4 \pi^2 xy = q_1 q_2 t^2 \).
  Then
  \begin{equation} \label{2426}
    L_{\chi_1, \chi_2}(s) = \sum_{n = 1}^\infty \frac{ \tau_{\chi_1, \chi_2}(n) }{n^s} V\left( \frac nx \right) + \alpha_{\chi_1, \chi_2}(s) \sum_{n = 1}^\infty \frac{ \tau_{ \overline{\chi_1}, \overline{\chi_2} }(n) }{ n^{1 - s} } V\left( \frac ny \right) + R_{\chi_1, \chi_2}(s; x, y),
  \end{equation}
  where \( R_{\chi_1, \chi_2}(s; x, y) \) satisfies the following individual bound,
  \begin{equation} \label{2440}
    R_{\chi_1, \chi_2}(s; x, y) \ll (q_1 q_2)^{ \frac{ 3 (1 - \sigma) }8 } t^{ -\frac{1 + 3\sigma}4 + \eps },
  \end{equation}
  as well as, for~\( T \gg \max\{ q_1, q_2 \} \), the following bound on average on the critical line,
  \begin{equation} \label{2416}
    \int_{T/2}^T \left| R_{\chi_1, \chi_2}\left( \tfrac12 + \i t; \tfrac{ t \sqrt{q_1 q_2} }{2\pi}, \tfrac{ t \sqrt{q_1 q_2} }{2\pi} \right) \right| \, \d t \ll T^\eps.
  \end{equation}
  The implicit constants depend at most on~\( V \) and~\(\eps\).
\end{theorem}
\begin{proof}
  In the special case~\( q_1 = q_2 = 1 \), this result is proven in~\cite[Theorem~4.2]{Ivi91}.
  The proof can be adapted to our situation without any difficulties via Theorems \ref{221}~and~\ref{222}.
\end{proof}

A similar approximative formula holds for the second sum on the right hand side in~\eqref{2426}.

\begin{theorem} \label{242}
  Let~\(\eps > 0 \) and~\( \rho > 1 \).
  Let~\( V : (0, \infty) \to [0, \infty) \), \( s \in \CC \) and~\( x, y \geq 1 \) be as in \cref{241}.
  Then
  \begin{align*}
    \alpha_{\chi_1, \chi_2}(s) \sum_{n = 1}^\infty \frac{ \tau_{ \overline{\chi_1}, \overline{\chi_2} }(n) }{ n^{1 - s} } V\left( \frac ny \right) &= \sum_{n = 1}^\infty \frac{ \tau_{\chi_1, \chi_2}(n) }{n^s} V\left( \frac xn \right) V\left( \frac n{\rho x} \right) \\
      &\phantom{ = {} } + \alpha_{\chi_1, \chi_2}(s) \sum_{n = 1}^\infty \frac{ \tau_{ \overline{\chi_1}, \overline{\chi_2} }(n) }{ n^{1 - s} } V\left( \frac ny \right) V\left( \frac{\rho n}y \right) + R_{\chi_1, \chi_2}'(s; x, y),
  \end{align*}
  where~\( R_{\chi_1, \chi_2}'(s; x, y) \) satisfies the bounds \eqref{2440}~and~\eqref{2416}.
\end{theorem}
\begin{proof}
  As above, this can be proven by adapting the proof given in~\cite[Theorem~4.2]{Ivi91}.
\end{proof}

\section{Background on automorphic forms} \label{3}

The aim of this section is to briefly present the tools coming from the spectral theory of automorphic forms needed in the treatment of the shifted convolution problem in \cref{4}.
Apart from the well-known Kuznetsov formula, this in particular includes a certain variant of the large sieve inequalities for Fourier coefficients of automorphic forms.

For a general account of the theoretic background we refer to~\cite{DFI02} and~\cite{Iwa02}.
In our specific situation we will however rely mainly on the results worked out in~\cite{Dra17}.

\subsection{Fourier coefficients of automorphic forms} \label{31}

Let \(q\)~and~\(q_0\) be positive integers such that~\( q_0 \mid q \).
In the following, \( \psi \)~will always denote a Dirichlet character mod~\(q_0\).
Let~\( \kappa(\psi) \) be defined as in~\eqref{2109}.
Furthermore, it will be convenient to set
\[ i(\gamma, z) := cz + d \qquad \text{and} \qquad j(\gamma, z) := \frac{ c z + d }{ |cz + d| } \qquad \text{for} \qquad \gamma = \begin{pmatrix} a & b \\ c & d \end{pmatrix} \in \GL_2(\RR). \]

Let~\( \theta_k(q, \psi) \) be the dimension of the space of holomorphic cusp forms of weight~\( k \equiv \kappa(\psi) \bmod 2 \) with respect to~\( \Gamma_0(q) \) and with nebentypus~\(\psi\).
Let~\( f_{j, k}^\psi \),~\( 1 \leq j \leq \theta_k(q, \psi) \), be an orthonormal basis for this space.
Given a singular cusp~\( \mf a \) with associated scaling matrix~\( \sigma_\mf a \), we write the Fourier expansion of~\( f_{j, k}^\psi \) around~\( \mf a \) as
\[ i( \sigma_{\mf a}, z )^{-k} f_{j, k}^\psi( \sigma_{\mf a} z) = \frac{ (4\pi)^\frac k2 }{ \sqrt{ (k - 1)! } } \sum_{n = 1}^\infty \lambda_{j, k}^\psi(n, \mf a) n^\frac{k - 1}2 e(nz). \]

Next, let~\( u_j^\psi \),~\( j \geq 1 \), be an orthonormal basis of the space of Maa{\ss} cusp forms of weight~\( \kappa(\psi) \) with respect to~\( \Gamma_0(q) \) and with nebentypus~\( \psi \).
We can assume that each~\( u_j^\psi \) is either even or odd.
We denote the corresponding spectral parameters by~\( t_j^\psi \), and we write the Fourier expansion of~\(u_j^\psi\) around a singular cusp~\( \mf a \) as
\[ j( \sigma_{\mf a}, z )^{ -\kappa(\psi) } u_j^\psi( \sigma_{\mf a} z) = \sqrt{ \cosh(\pi t_j^\psi) } \sum_{n \neq 0} \rho_j^\psi( n, \mf a ) n^{-\frac12} W_{ \frac n{ |n| } \frac{ \kappa(\psi) }2, \i t_j^\psi }(4 \pi |n| y) e(nx), \]
where \( W_s(\xi) \) denotes the Whittaker function as defined in~\cite[(1.26)]{Iwa02}.
Note that we can choose the spectral parameters in such a way that either~\( t_j^\psi \in [0, \infty) \) or~\( \i t_j^\psi \in [0, \infty) \).
The spectral parameters which satisfy the latter condition are called exceptional.
It is widely believed that Maa{\ss} cusp forms with exceptional spectral parameter do not exist, although this has not been proven so far.
Let~\( \theta \in [0, \infty) \) be such that \( \i t_j^\psi \leq \theta \) for all exceptional~\( t_j^\psi \), uniformly for all levels~\(q\) and any nebentypus~\(\psi\).
By the work of Kim and Sarnak~\cite{Kim03}, we know that the value
\[ \theta = \frac 7{64} \]
is admissible.

Last but not least, we write the Fourier expansion of the Eisenstein series~\( E_\mf c^\psi( z; 1/2 + \i t ) \) of weight~\( \kappa(\psi) \) with respect to~\( \Gamma_0(q) \) and with nebentypus~\( \psi\), associated to the singular cusp~\( \mf c \), around a singular cusp~\( \mf a \) as
\begin{multline*}
  j(\sigma_\mf a, z)^{ -\kappa(\psi) } E_\mf c^\psi\left( \sigma_{\mf a} z; \tfrac12 + \i t \right) = c_{ \mf c, 1 }^\psi(t) y^{\frac12 + \i t} + c_{ \mf c, 2 }^\psi(t) y^{\frac12 - \i t} \\
    + \sqrt{ \cosh(\pi t) } \sum_{n \neq 0} \phi_{ \mf c, t }^\psi( n, \mf a ) n^{-\frac12} W_{ \frac n{ |n| } \frac{ \kappa(\psi) }2, \i t  }(4 \pi |n| y) e(nx).
\end{multline*}

Note that the normalization of the Fourier coefficients used here differs from the one used in~\cite{Dra17} and~\cite{Top17}, from where we will cite some results further below.

\subsection{Bounds for Kloosterman sums} \label{32}

Let \( \mf a \)~and~\( \mf b \) be cusps of~\( \Gamma_0(q) \) which are singular with respect to all characters~\( \psi \)~mod~\( q_0 \), and let \( \sigma_\mf a \)~and~\( \sigma_\mf b \) be their associated scaling matrices.
For~\( m, n \in \ZZ \) and~\( c \in (0, \infty) \) the Kloosterman sum associated to~\( \mf a \) and~\( \mf b \) is defined as
\[ S_{\mf a \mf b}^\psi(m, n; c) := \sum_{ d \bmod c \ZZ } \overline\chi\left( \sigma_\mf a \begin{pmatrix} a & b \\ c & d \end{pmatrix} {\sigma_\mf b}^{-1} \right) e\left( m \frac ac + n \frac dc \right), \]
where the sum runs over all~\( d \)~mod~\( c \ZZ \) for which there exist \(a\)~and~\(b\) such that
\[ \begin{pmatrix} a & b \\ c & d \end{pmatrix} \in {\sigma_\mf a}^{-1} \Gamma_0(q) \sigma_\mf b. \]
Note that this definition depends on the particular choice of the associated scaling matrices.
Furthermore, depending on the choice of~\(c\) the sum may well be empty.

Of particular importance are the sums with~\( \mf a = \mf b \), as they come up in the proof of the large sieve inequalities.
In the following, we will focus on a particular set of cusps~\( \mf a \), namely
\[ \mf A := \{ \infty \} \cup \left\{ u/w \in \QQ : u, w \in \ZZ_{\geq 1}, \,\, (u, w) = 1, \,\, w \mid q, \,\, \left( w, q/w \right) = 1 \right\}, \]
since they are easier to work with from a technical point of view, and since they cover all the cases we need.
Note that all the cusps in~\( \mf A \) are singular with respect to all characters mod~\( q_0 \).

As can be deduced from~\cite[Lemma~4.1]{Dra17}, the sum~\( S_{\mf a \mf a}^\psi(m, n; c) \) for~\( \mf a \in \mf A \) is non-empty exactly when \( c \)~is an integer divisible by~\(q\), in which case we have
\begin{equation} \label{3214}
  \left| S_{\mf a \mf a}^\psi(m, n; c) \right| = \left| S_\psi(m, n; c) \right|,
\end{equation}
where \( S_\psi(m, n; c) \) is the usual twisted Kloosterman sum,
\[ S_\psi(m, n; c) := \sum_\twoln{a \bmod c}{ (a, c) = 1 } \psi(a) e\left( \frac{ ma + n \overline a }c \right). \]
Concerning upper bounds, we know by~\eqref{3214} and~\cite[Theorem~9.2]{KL13} that
\[ S_{\mf a \mf a}^\psi(m, n; c) \ll (m, n, c)^\frac12 ( q_0 c )^{\frac12 + \eps}. \]
The factor~\( {q_0}^{1/2} \) appearing on the right hand side is unfavorable, but in general cannot be omitted (see~\cite[Example~9.9]{KL13}).
However, it effectively disappears if we include a further averaging over all characters~\(\psi\)~mod~\(q_0\).

\begin{lemma} \label{3224}
  Let~\( \eps > 0 \).
  Let \(c\)~and~\(q_0\) be positive integers such that~\( q_0 \mid c \), let~\( m, n \in \ZZ \) and let \(\mf a \in \mf A \).
  Then
  \begin{equation} \label{3265}
    \frac1{ \phi(q_0) } \sum_{\psi \bmod q_0} \left| S_{\mf a \mf a}^\psi(m, n; c) \right|^2 \ll (m, n, c) c^{1 + \eps},
  \end{equation}
  where the implicit constant depends only on~\( \eps \).
\end{lemma}
\begin{proof}
  By~\eqref{3214} it is enough to consider the case of usual twisted Kloosterman sums.
  Moreover, by twisted multiplicativity of Kloosterman sums it is enough to consider the case when \(c\)~and~\(q_0\) are powers of a prime~\(p\).
  Hence, let~\( c = p^\ell \) and~\( q_0 = p^{\ell_0} \) with~\( \ell_0 \leq \ell \), and let \(k\) be the largest integer such that~\( p^k \mid (m, n) \).
  
  If~\( \ell_0 < \ell - k \), then by~\cite[Propositions 9.7~and~9.8]{KL13} we have
  \[ S_\psi(m, n; c) \ll p^{ \frac{k + \ell}2 + \eps } = (m, n, c)^\frac12 c^{\frac12 + \eps}, \]
  and \eqref{3265}~follows immediately.
  If~\( \ell_0 \geq \ell - k \), then we have by orthogonality of Dirichlet characters,
  \[ \frac1{ \phi(q_0) } \sum_{\psi \bmod q_0} \left| S_\psi( m, n; c ) \right|^2 = \sum_\twoln{a_1, a_2 \bmod c}{ a_1 \equiv a_2 \bmod q_0 } e\left( \frac{ (a_1 - a_2) m + ( \overline{a_1} - \overline{a_2} ) n }c \right) = \sum_\twoln{a_1, a_2 \bmod c}{ a_1 \equiv a_2 \bmod q_0 } 1, \]
  so that
  \[ \frac1{ \phi(q_0) } \sum_{\psi \bmod q_0} \left| S_\psi( m, n; c ) \right|^2 \leq \frac{c^2}{q_0} = p^{2\ell - \ell_0} \leq p^{k + \ell} = (m, n, c) c, \]
  and we see that \eqref{3265}~also holds in this case.
\end{proof}

\subsection{The Kuznetsov formula} \label{33}

Let~\( f : (0, \infty) \to \CC \) be a smooth and compactly supported function.
Given a Dirichlet character~\( \psi \)~mod~\(q_0\), we define the following integral transforms of~\(f\),
\begin{align}
  \tilde f(t) &:= \frac{ 2\pi \i t^{ \kappa(\psi) } }{ \sinh(\pi t) } \int_0^\infty \! \left( J_{2\i t}(\eta) - (-1)^{ \kappa(\psi) } J_{-2\i t}(\eta) \right) f(\eta) \, \frac{d\eta}\eta, \label{3393} \\
  \check f(t) &:= 8 \i^{ -\kappa(\psi) } \cosh(\pi t) \int_0^\infty \! K_{2\i t}(\eta) f(\eta) \, \frac{d\eta}\eta, \label{3389} \\
  \dot f(k) &:= 4 \i^k \int_0^\infty \! J_{k - 1}(\eta) f(\eta) \, \frac{d\eta}\eta. \label{3382}
\end{align}
Note that these integral transforms depend on the parity of the character~\(\psi\), even though we do not indicate this in the notation.

The Kuznetsov formula then reads as follows (see~\cite[Theorem~2.3]{Top17}).

\begin{theorem} \label{331}
  Let~\( f : (0, \infty) \to \mathbb{C} \) be a smooth and compactly supported function, let~\( \mf a, \mf b \in \mf A \), let \( \psi \)~mod~\(q_0\) be a Dirichlet character, and let~\(m\),~\(n\) be positive integers.
  Then
  \begin{align*}
    \sum_c \frac{ S_{\mf a \mf b}^\psi(m, n; c) }c f\left( 4\pi \frac{ \sqrt{mn} }c \right) &= \sum_{j = 1}^\infty \tilde f(t_j^\psi) \overline{ \rho_j^\psi(m, \mf a) } \rho_j^\psi(n, \mf b) + \!\!\!\! \sum_\thrln{ k > \kappa(\psi) }{ k \equiv \kappa(\psi) \bmod 2 }{ 1 \leq j \leq \theta_k(q, \psi) } \!\!\!\! \dot f(k) \overline{ \lambda_{j, k}^\psi (m, \mf a) } \lambda_{j, k}^\psi( n, \mf b ) \\
        &\phantom{ = {} } + \sum_{ \mf c \text{ sing.} } \frac1{4\pi} \int_{-\infty}^\infty \! \tilde f(t) \overline{ \phi_{ \mf c, t }^\psi ( m, \mf a ) } \phi_{\mf c, t }^\psi ( n, \mf b ) \, \d t, \\
    \intertext{and}
    \sum_c \frac{ S_{\mf a \mf b}^\psi(m, -n; c) }c f\left( 4\pi \frac{ \sqrt{mn} }c \right) &= \sum_{j = 1}^\infty \check f(t_j^\psi) \overline{ \rho_j^\psi(m, \mf a) } \rho_j^\psi(-n, \mf b) \\
        &\phantom{ = {} } + \sum_{ \mf c \text{ sing.} } \frac1{4\pi} \int_{-\infty}^\infty \! \check f(t) \overline{ \phi_{ \mf c, t }^\psi ( m, \mf a ) } \phi_{\mf c, t }^\psi ( -n, \mf b ) \, \d t,
  \end{align*}
  where \(c\) runs over all positive real numbers for which \( S_{\mf a \mf b}^\psi(m, \pm n; c) \)~is non-empty.
\end{theorem}

Assume that \(q\)~is of the form~\( q = rs \) for positive coprime integers~\(r\) and~\(s\) with~\( q_0 \mid r \).
If we consider the cusps~\( \mf a = \infty \) and~\( \mf b = 1/s \), together with associated scaling matrices
\[ \sigma_\infty = \begin{pmatrix} 1 & 1 \\ & 1 \end{pmatrix} \qquad \text{and} \qquad \sigma_{1/s} = \begin{pmatrix} \sqrt{r} & 1 \\ s \sqrt{r} & \sqrt{r}^{-1} \end{pmatrix}, \]
then the left hand sides of the two formulae in \cref{331} become
\begin{equation} \label{3343}
  \sum_c \frac{ S_{\infty 1/s}^\psi(m, \pm n; c) }c f\left( 4\pi \frac{ \sqrt{mn} }c \right) = e\left( \frac{ \pm n \overline{s} }r  \right) \sum_{ (c, r) = 1 } \overline \psi(c) \frac{ S\left( m, \pm n \overline r; sc \right) }{ \sqrt{r} sc } f\left( 4\pi \frac{ \sqrt{mn} }{ \sqrt{r} sc } \right).
\end{equation}
It is in this specific form that we will use the Kuznetsov formula in \cref{4}.

\subsection{Large sieve inequalities} \label{34}

The aim of this section is to deduce a variant of the large sieve inequalities for Fourier coefficients of automorphic forms adapted to our specific setting.
We could in principle use~\cite[Proposition~4.7]{Dra17}, however the factor~\( {q_0}^{1/2} \) appearing there is disadvantageous in our situation.
As we will show, this factor can be removed by averaging over all~\( \psi \)~mod~\(q_0\).

Let~\( \mf a \in \mf A \) and~\( N \geq 1 \).
For each \( \psi \)~mod~\(q_0\), let~\( a_n^\psi \) be a sequence of complex numbers supported in~\( N/2 < n \leq N \), and set
\[ \| a_n^\psi \| := \sum_{ N/2 < n \leq N } \max_{\psi \bmod q_0} | a_n^\psi |^2. \]
Furthermore, let
\begin{align*}
  \Sigma_{1, \pm}^\psi (j) &:= \sum_{ N/2 < n \leq N } a_n^\psi \rho_j^\psi( \pm n, \mf a ), \qquad \Sigma_{2, \pm}^\psi (\mf c, t) := \sum_{ N/2 < n \leq N } a_n^\psi \phi_{ \mf c, t }^\psi( \pm n, \mf a ), \\
  \Sigma_3^\psi (j, k) &:= \sum_{ N/2 < n \leq N } a_n^\psi \lambda_{j, k}^\psi( n, \mf a ).
\end{align*}

Then the following variant of the large sieve inequalities holds.

\begin{theorem} \label{341}
  Let~\( \eps > 0 \).
  Let~\( T, N \geq 1 \) and~\( \mf a \in \mf A \).
  Let~\( a_n^\psi \) be as described above.
  Then
  \begin{align*}
    \frac1{ \phi(q_0) } \sum_{\psi \bmod q_0} \sum_{ | t_j^\psi | \leq T } ( 1 + |t_j^\psi| )^{ \pm \kappa(\psi) } \left| \Sigma_{1, \pm}^\psi (j) \right|^2 &\ll \left( T^2 + \frac{ N^{1 + \varepsilon} }q \right) \| a_n^\psi \|, \\
    \frac1{ \phi(q_0) } \sum_{\psi \bmod q_0} \sum_{ \mf c \text{ sing.} } \int_{-T}^T \! ( 1 + |t| )^{ \pm \kappa(\psi) } \left| \Sigma_{2, \pm}^\psi (\mf c, t) \right|^2 \, \d t &\ll \left( T^2 + \frac{ N^{1 + \varepsilon} }q \right) \| a_n^\psi \|, \\
    \frac1{ \phi(q_0) } \sum_{ \psi \bmod q_0 } \sum_\twoln{ \kappa(\psi) < k \leq T }{ k \equiv \kappa(\psi) \bmod 2 } \sum_{ 1 \leq j \leq \theta_k(q, \psi) } \left| \Sigma_3^\psi (j, k) \right|^2 &\ll \left( T^2 + \frac{ N^{1 + \varepsilon} }q \right) \| a_n^\psi \|,
  \end{align*}
  where the implicit constants depend only on~\( \varepsilon \).
\end{theorem}
\begin{proof}
  The proof is in large parts identical to the proof of the original large sieve inequalities as given by Deshouillers and Iwaniec~\cite[Theorem~2]{DI82}, and its generalization to arbitrary nebentypus as worked out by Drappeau~\cite[Proposition~4.7]{Dra17}.
  We will therefore restrict ourselves to pointing out the main differences.
  
  Let~\( \kappa_0 \in \{ 0, 1 \} \), \( \vartheta \in (0, \infty) \) and~\( \lambda \in [0, \infty) \), and set
  \[ B_\mf a(\lambda, \vartheta, c, N) := \frac1{ \phi(q_0) } \sum_\twoln{\psi \bmod q_0}{ \kappa(\psi) = \kappa_0 } \sum_\twoln{ N/2 < n_1 \leq N }{ N/2 < n_2 \leq N } a_{n_1}^\psi \overline{ a_{n_2}^\psi } e^{ -\lambda \sqrt{n_1 n_2} } S_{\mf a \mf a}^\psi(n_1, n_2, c) e\left( 2 \frac{ \sqrt{n_1 n_2} }c \vartheta \right). \]
  Then we have the following bounds for this expression,
  \begin{align}
    | B_\mf a(\lambda, \vartheta, c, N) | &\ll c^{\frac12 + \eps} N \| a_n^\psi \|, \label{3402} \\
    | B_\mf a(\lambda, \vartheta, c, N) | &\ll ( c + N + \sqrt{\vartheta c N} ) \| a_n^\psi \|, \label{3446} \\
    | B_\mf a(\lambda, \vartheta, c, N) | &\ll \vartheta^{-\frac12} c^\frac12 N^{\frac12 + \eps} \| a_n^\psi \| \qquad\qquad \text{(for \( \vartheta < 2 \) and~\( c < N \)),} \label{3439}
  \end{align}
  with all the implicit constants depending at most on~\(\eps\).
  Here the first bound~\eqref{3402} is a direct consequence of \cref{3224}, while \eqref{3446}~and~\eqref{3439} are proven in~\cite[Lemma~4.6]{Dra17}.
  
  From this point on we can follow the proof of~\cite[Proposition~4.7]{Dra17}, always taking into account the extra summation over~\( \psi \).
  We leave the details to the reader.
\end{proof}

When taking care of the exceptional eigenvalues, the following weighted large sieve inequality will be useful.

\begin{theorem} \label{342}
  Let~\( \eps > 0 \).
  Let~\( 1 \leq N \leq q^2 \) and~\( \mf a \in \mf A \).
  Let~\( a_n^\psi \) be as described above.
  Then
  \[ \frac1{ \phi(q_0) } \sum_{\psi \bmod q_0} \sum_{ \text{\(t_j^\psi\)~exc.} } \left( \frac q{ N^\frac12 } \right)^{4 \i t_j^\psi} \left| \Sigma_{1, \pm}^\psi ( j ) \right|^2 \ll q^\eps N^{1 + \eps} \| a_n^\psi \|, \]
  where the implicit constant depends only on~\( \varepsilon \).
\end{theorem}
\begin{proof}
  This result is a direct consequence of the Cauchy-Schwarz inequality and the following estimate,
  \[ \frac1{ \phi(q_0) } \sum_\twoln{ \psi \bmod q_0 }{ \kappa(\psi) = \kappa_0 } \sum_{ \text{\(t_j^\psi\) exc.} } \left( \frac q{ n^\frac12 } \right)^{4\i t_j^\psi} \big| \rho_j^\psi(\pm n, \mf a) \big|^2 \ll (qn)^\eps (q, n)^\frac12, \]
  where~\( \kappa_0 \in \{ 0, 1 \} \).
  It can be proven in the same way as~\cite[(16.58)]{IK04} with the difference that in order to bound the Kloosterman sums, \cref{3224} has to be used instead of Weil's bound.
\end{proof}

\section{A shifted convolution problem} \label{4}

In this section, we consider the shifted convolution problem which is at the heart of the proof of Theorems~\ref{101}--\ref{106}.

As usual, let \( \chi_1 \)~mod~\(q_1\) and \( \chi_2 \)~mod~\(q_2\) be primitive Dirichlet characters, and set~\( q_1^\ast := \left( q_1, {q_2}^\infty \right) \) and~\( q_2^\ast := \left( q_2, {q_1}^\infty \right) \).
Furthermore, let~\( \delta > 0 \) be a fixed constant, let~\( \alpha, N, H \geq 1 \) be real numbers satisfying the condition
\begin{equation} \label{4021}
  \alpha^\frac23 H \leq N^{1 - \delta},
\end{equation}
and let~\( f : (0, \infty) \times \RR \to \CC \) be a smooth weight function, compactly supported in either
\[ \supp f \subset [ N/4, 2 N ] \times [ H/4, 2 H ] \qquad \text{or} \qquad \supp f \subset [ N/4, 2 N ] \times [ -2H, -H/4 ], \]
and with derivatives satisfying the bounds
\begin{equation} \label{4096}
  \frac{ \partial^{\nu_1 + \nu_2} }{ \partial \xi^{\nu_1} \partial\eta^{\nu_2} } f(\xi, \eta) \ll \frac1{ N^{\nu_1} H^{\nu_2} } \quad \text{for} \quad \nu_1, \nu_2 \geq 0.
\end{equation}
We are then interested in the following shifted convolution sum
\[ D_{\chi_1, \chi_2}(f, \alpha) := \sum_h \frac1h \sum_n \tau_{\chi_1, \chi_2}(n) \tau_{ \overline{\chi_1}, \overline{\chi_2} }(n + h) f(n, h) e\left( \alpha \frac hn \right), \]
and our aim will be to prove the following asymptotic formula.

\begin{proposition} \label{401}
  Let~\( \delta, \eps > 0 \).
  Let~\( f, \alpha, N, H \) be as described above.
  Then
  \begin{multline*}
    D_{\chi_1, \chi_2}(f, \alpha) = \sum_h \frac1h \int \! Q_{\chi_1, \chi_2}( \log\xi, \log(\xi + h); h ) f(\xi, h) e\left( \alpha \frac h\xi \right) \, \d\xi \\
      + \LO{ ( q_1^\ast q_1 + q_2^\ast q_2 )^\frac12 [q_1, q_2]^\frac12 N^{\frac12 + \eps} \bigg( 1 + \alpha \frac{ H^\frac12 }N \bigg) + \left( \frac{ q_1^\ast q_1 }{ {q_1^\ast}^{4\theta} } + \frac{ q_2^\ast q_2 }{ {q_2^\ast}^{4\theta} } \right)^\frac12 [q_1, q_2]^{\frac12 - 2\theta} N^{\frac12 + \theta + \eps} },
  \end{multline*}
  where~\( Q_{\chi_1, \chi_2}(X_1, X_2; h) \) is a polynomial in \(X_1\)~and~\(X_2\) of degree at most~\(2\) with coefficients depending only on~\(\chi_1\), \(\chi_2\) and~\(h\).
  The implicit constant depends at most on~\( \delta \), \( \eps \) and the implicit constants in~\eqref{4096}.
\end{proposition}

Remember that \(\theta\) denotes the bound in the Ramanujan-Petersson conjecture (see \cref{31}).
Here we are only concerned with the evaluation of the sum over~\(n\), while we will take care of the remaining sum over~\(h\) at a later stage.
Nevertheless, the additional average over~\(h\) will simplify some of the estimations in the proof.

The polynomial~\( Q_{\chi_1, \chi_2}(X_1, X_2; h) \) can be stated in fairly explicit terms.
Let
\begin{equation} \label{4086}
  \psi_z(q) := \sum_{d \mid q} \frac{ \mu(d) }{ d^{1 + z} }, \qquad Z_q(z) := \psi_z(q) z \zeta(z + 1) \qquad \text{and} \qquad \Delta_z := \frac\partial{\partial z} \bigg|_{z = 0}.
\end{equation}
Then, if~\( \chi_1 = \chi_2 \), it is the quadratic polynomial given by
\begin{equation} \label{4065}
  Q_{\chi_1, \chi_2}(\log\xi_1, \log\xi_2; h) := \Delta_{z_1} \Delta_{z_2} {\xi_1}^{z_1} {\xi_2}^{z_2} Z_{q_1}(2 z_1) Z_{q_2}(2 z_2) \frac{ r_{q_1}(h) }{q_1} \sum_\twoln{ c = 1 }{ (c, q_1) = 1 }^\infty \frac{ r_c(h) }{ c^{2 + 2 z_1 + 2 z_2} },
\end{equation}
while if~\( \chi_1 \neq \chi_2 \), it is simply a constant, namely, in the case~\( q_1 = q_2 \),
\begin{align*}
  &Q_{\chi_1, \chi_2}(X_1, X_2; h) := 2 \left| L( 1, \chi_1 \overline{\chi_2} ) \right|^2 \frac{ r_{q_1}(h) }{q_1} \sum_\twoln{c = 1}{ (c, q_1) = 1 }^\infty \frac{ r_c(h) }{ c^2 } \\
    &\quad + L( 1, \overline{\chi_1} \chi_2 )^2 \frac{ G( \overline{\chi_1} \chi_2, h) }{ \overline{ G(\chi_1) } G( \chi_2 ) } \sum_{c = 1}^\infty \frac{ r_c(h) ( \overline{\chi_1} \chi_2 )^2(c) }{ c^2 } + L( 1, \chi_1 \overline{\chi_2} )^2 \frac{ G( \chi_1 \overline{\chi_2}, h) }{ G(\chi_1) \overline{ G( \chi_2 ) } } \sum_{c = 1}^\infty \frac{ r_c(h) ( \chi_1 \overline{\chi_2} )^2(c) }{ c^2 },
\end{align*}
and, in the case~\( q_1 \neq q_2 \),
\[ Q_{\chi_1, \chi_2}(X_1, X_2; h) := \left| L(1, \overline{\chi_1} \chi_2 ) \right|^2 \sum_\twoln{c = 1}{ (c, q_1) = 1 }^\infty \frac{ r_{c q_2}(h) }{c^2 q_2} + \left| L( 1, \chi_1 \overline{\chi_2} ) \right|^2 \sum_\twoln{c = 1}{ (c, q_2) = 1 }^\infty \frac{ r_{c q_1}(h) }{c^2 q_1}. \]

\subsection{Initial transformations} \label{41}

Let~\( u : (0, \infty) \to [0, 1] \) be a smooth and compactly supported weight function which satisfies the conditions
\begin{equation} \label{4195}
  \supp u \subset [ 1/2, 2 ] \qquad \text{and} \qquad \sum_{j \in \ZZ} u\left( \frac\xi{2^j} \right) = 1 \quad \text{for} \quad \xi \in (0, \infty).
\end{equation}
We set
\[ u_0(\xi) := \sum_{i \leq 0} u\left( \frac{8\xi}{2^i \sqrt{N} } \right) \qquad \text{and} \qquad u_j(\xi) = u\left( \frac{8\xi}{ 2^j \sqrt{N} } \right) \quad \text{for} \quad j \geq 1. \]
We start the proof of \cref{401} by opening the divisor function~\( \tau_{ \overline{\chi_1}, \overline{\chi_2} }(n) \) and localizing the two new variables in dyadic intervals via the smooth partition of unity defined above.
This way our original sum~\( D_{\chi_1, \chi_2}(f, \alpha) \) is split up into the sums
\begin{equation} \label{4192}
  D_{j_1, j_2} := \sum_{n_1, n_2, h} \overline{\chi_1}(n_1) \overline{\chi_2}(n_2) \tau_{\chi_1, \chi_2}(n_1 n_2 - h) u_{j_1}(n_1) u_{j_2}(n_2) \frac{ f(n_1 n_2 - h, h) }h e\left( \alpha \frac h{n_1 n_2 - h} \right),
\end{equation}
with \(j_1\)~and~\(j_2\) ranging over~\( 0 \leq j_1, j_2 \ll \log N \).
Note that~\( D_{0, 0} \) is empty.

Since the expression~\eqref{4192} is symmetric in \(n_1\)~and~\(n_2\), we can assume without loss of generality that~\( j_2 \geq 1 \).
The variables \(n_1\)~and~\(n_2\) are then supported in the ranges
\[ n_1 \asymp N_1 \qquad \text{and} \quad \quad n_2 \asymp N_2 \qquad \text{with} \qquad N_2 := 2^{j_2 - 3} N^{1/2}, \quad N_1 := N / N_2, \]
and we have~\( N_1 \ll N^{1/2} \ll N_2 \).

In~\( D_{j_1, j_2} \) we split the variable~\(n_2\) into residue classes modulo~\(q_2\), so that the sum becomes
\[ D_{j_1, j_2} = \sum_{n_1, h} \overline{\chi_1}(n_1) \sum_{a_2 \bmod q_2} \overline{\chi_2}(a_2) \sum_{ m \equiv n_1 a_2 - h \bmod n_1 q_2 } \tau_{\chi_1, \chi_2}(m) g_{n_1, h}(m), \]
with
\[ g_{n_1, h}(\xi) := u_{j_1}(n_1) u_{j_2}\left( \frac{ \xi + h }{n_1} \right) \frac{ f( \xi, h ) }h e\left( \alpha \frac h\xi \right). \]
At this point, we use \cref{232} to evaluate the sum over~\(m\), and get
\[ D_{j_1, j_2} = \Sigma_{j_1, j_2}^0 + \Sigma_{j_1, j_2}^+ + \Sigma_{j_1, j_2}^-, \]
where~\( \Sigma_{j_1, j_2}^0 \) takes the form
\begin{align}
  \Sigma_{j_1, j_2}^0 &:= \frac1{ q_2 } \sum_\twoln{c, n, h}{c \mid n q_2} \frac{ \overline{\chi_1}(n) }{cn} \sum_\twoln{a_0 \bmod c}{ (a_0, c) = 1 } e\left( \frac{ a_0 h }c \right) G\left( \overline{\chi_2}, -a_0 \frac{n q_2}c \right) \int \! \Pi_{\chi_1, \chi_2}( \log\xi; c, a_0 ) g_{n, h}(\xi) \, \d\xi, \label{4164} \\
  \intertext{and where the other two sums are given by}
  \Sigma_{j_1, j_2}^\pm &:= \frac1{ q_2 } \sum_{n, h} \frac{ \overline{\chi_1}(n) }n \sum_{ c \mid n q_2 } \sum_{m = 1}^\infty \frac{ (c q_2)^\frac12 K_{\chi_1, \chi_2}^\pm(m, n, h, c) }{ [c, q_1]^\frac12 [c, q_2]^\frac12 } \int \! B_{\chi_1, \chi_2}^\pm\left( \frac{ (m \xi)^\frac12 }{ [c, q_1]^\frac12 [c, q_2]^\frac12 } \right) g_{n, h}(\xi) \, \d\xi, \nonumber
\end{align}
with
\begin{equation} \label{4125}
  K_{\chi_1, \chi_2}^\pm(m, n, h, c) := \frac1{ {q_2}^\frac12 } \sum_{a_2 \bmod q_2} \overline{\chi_2}(a_2) T_{\chi_1, \chi_2}( m; c, \pm ( n a_2 - h ) ).
\end{equation}

As we will show in \cref{45}, the contribution coming from the terms~\( \Sigma_{j_1, j_2}^0 \) together forms the main term in \cref{401}.
Before coming to that, we will however first take care of the other two sums~\( \Sigma_{j_1, j_2}^\pm \).
Once more it will be advantageous to localize the variable~\(m\) in a dyadic interval, so instead of looking at these sums directly, we will consider
\[ \Sigma_{j_1, j_2}^\pm(M) := \sum_{m, n, h} \frac{ \overline{\chi_1}(n) }n u\left( \frac mM \right) \sum_{ c \mid n q_2 } \frac{ c^\frac12 K_{\chi_1, \chi_2}^\pm(m, n, h, c) }{ {q_2}^\frac12 [c, q_1]^\frac12 [c, q_2]^\frac12 } \int \! B_{\chi_1, \chi_2}^\pm\left( \frac{ (m \xi)^\frac12 }{ [c, q_1]^\frac12 [c, q_2]^\frac12 } \right) g_{n, h}(\xi) \, \d\xi, \]
with the weight function~\(u\) as defined in~\eqref{4195}.

\subsection{Evaluation of~\texorpdfstring{\( K_{\chi_1, \chi_2}^\pm(m, n, h, c) \)}{K(m, n, h, c)}} \label{42}

Before going any further, we first need to evaluate the exponential sum~\eqref{4125} and express it in terms of Kloosterman sums.
This will allow us afterwards to make use of the Kuznetsov formula.

We decompose the moduli \(q_1\)~and~\(q_2\) as follows,
\[ q_1^\ast := \left( q_1, {q_2}^\infty \right), \quad q_2^\ast := \left( q_2, {q_1}^\infty \right) \qquad \text{and} \qquad q_1^\circ := q_1 / q_1^\ast, \quad q_2^\circ := q_2 / q_2^\ast, \]
and accordingly write the Dirichlet characters \( \chi_1 \)~and~\( \chi_2 \) as
\[ \chi_i = \chi_i^\ast \chi_i^\circ \quad \text{with} \quad \chi_i^\ast \bmod q_i^\ast \quad \text{and} \quad \chi_i^\circ \bmod q_i^\circ. \]
Note that the characters \( \chi_1^\ast \), \( \chi_1^\circ \), \( \chi_2^\ast \) and~\( \chi_2^\circ \) are all primitive.
We also set
\[ h = h^\ast h^\circ \quad \text{with} \quad h^\ast := (h, q_2^\ast) \quad \text{and} \quad h^\circ := h / h^\ast. \]
Furthermore, we define the quantity
\[ \kappa_{\chi_1, \chi_2} := \chi_1^\ast \overline{\chi_2^\ast}(q_1^\circ q_2^\circ) \chi_1^\circ \overline{\chi_2^\circ}( [q_1^\ast, q_2^\ast] ) \frac{ G( \chi_1^\circ \overline{\chi_2^\circ } ) }{ \sqrt{ q_1^\circ q_2^\circ} }, \]
as well as the exponential sum
\[ E_{\chi_1, \chi_2}(m; \psi) := \frac{ \overline\psi( q_1^\circ {q_2^\circ}^2 ) \overline{ G(\psi) } }{ \sqrt{ q_2^\ast / h^\ast } } \! \sum_{m_1 m_2 = m} \! \frac{ \overline{\chi_1^\circ} \chi_2^\circ(m_1) G( \chi_1^\ast \overline{ \chi_2^\ast \psi }, m_1 ) }{ \sqrt{ q_2^\ast [q_1^\ast, q_2^\ast] } } \!\! \sum_{ a \bmod q_2^\ast } \psi \chi_2^\ast(a) \overline{\chi_2^\ast}( a + m_2 ), \]
where \(\psi\) is a Dirichlet character mod~\( q_2^\ast / h^\ast \).

With the necessairy notation set up, we can now state the main result of this section.

\begin{lemma} \label{421}
  The sum~\( K_{ \chi_1, \chi_2 }^\pm(m, n, h, c) \) vanishes unless~\( (c, q_1 q_2) = q_2 \), in which case we have
  \begin{multline*}
    K_{ \chi_1, \chi_2 }^\pm(m, n, h, c) = \chi_2(\mp 1) \chi_1(n) \kappa_{\chi_1, \chi_2} \frac{ {q_2^\ast}^\frac12 }{ {h^\ast}^\frac12 } \frac1{ \phi(q_2^\ast / h^\ast) } \sum_{ \psi \bmod q_2^\ast / h^\ast } \psi(\mp h^\circ) E_{\chi_1, \chi_2}(m; \psi) \\
      \cdot \overline{\chi_1} \chi_2 \left( \frac{n q_2}c \right) \overline{\psi^2}\left( \frac c{q_2} \right) \frac{ S( \mp h^\ast h^\circ, \overline{ q_2^\ast [q_1, q_2^\ast] } m ; c / q_2^\ast ) }{ (c / q_2^\ast)^\frac12 },
  \end{multline*}
  where \( \psi \)~runs over all Dirichlet characters mod~\( q_2^\ast / h^\ast \).
\end{lemma}
\begin{proof}
  Remember that~\( (n, q_1) = 1 \).
  Since~\( c \mid n q_2 \), the sum over~\( a_2 \) in~\eqref{4125} is simply a Gau{\ss} sum mod~\(q_2\), which can be evaluated directly.
  Hence \( K_{ \chi_1, \chi_2 }^\pm(m, n, h, c) \)~becomes
  \begin{equation} \label{4264}
    K_{\chi_1, \chi_2}^\pm(m, n, h, c) = \chi_2\left( \mp \frac{ n q_2 }c \right) \frac{ G( \overline{\chi_2} ) }{ {q_2}^\frac12 } \sum_{m_1 m_2 = m} \frac{ \tilde K_{\chi_1, \chi_2}(m_1, m_2, \pm h, c) }{ c^\frac12 [c, q_1]^\frac12 [c, q_2]^\frac12 },
  \end{equation}
  with
  \begin{equation} \label{4242}
    \tilde K_{\chi_1, \chi_2}(m_1, m_2, f, c) := \sum_\twoln{a \bmod c}{ (a, c) = 1 } \sum_\twoln{ b_1 \bmod [c, q_1] }{ b_2 \bmod [c, q_2] } \chi_1(b_1) \chi_2(a b_2) e\left( \frac{ a ( b_1 b_2 + f ) }c + \frac{m_1 b_1}{ [c, q_1] } + \frac{m_2 b_2}{ [c, q_2] } \right).
  \end{equation}
  In particular, we see that the sum vanishes unless~\( q_2 \mid c \).
  Moreover, we have~\( \left( c / q_2, q_1 \right) = 1 \).
  
  In view of this, we write the variable~\(c\) as
  \[ c = c_0 c_2 q_2 \quad \text{with} \quad c_2 := \left( c / q_2, {q_2}^\infty \right) \quad \text{and} \quad c_0 := c / ( c, {q_2}^\infty ). \]
  Note that with these definitions we have~\( (c_0, q_1 q_2) = (c_2, q_1) = 1 \).
  We write the variables \(a\),~\(b_1\)~and~\(b_2\) inside~\eqref{4242} accordingly as
  \begin{alignat*}{3}
    a &= a_0 c_2 q_2 + a_2 c_0 \qquad &\text{with} \qquad a_0 &\bmod c_0 \quad &\text{and} \quad a_2 &\bmod c_2 q_2, \\
    b_1 &= d_1 c_2 [q_1, q_2] + u_1 c_0 \qquad &\text{with} \qquad d_1 &\bmod c_0 \quad &\text{and} \quad u_1 &\bmod c_2 [q_1, q_2], \\
    b_2 &= d_2 c_2 q_2 + u_2 c_0 \qquad &\text{with} \qquad d_2 &\bmod c_0 \quad &\text{and} \quad u_2 &\bmod c_2 q_2.
  \end{alignat*}
  so that \( \tilde K_{\chi_1, \chi_2}(m_1, m_2, f, c) \) takes the form
  \begin{equation}
    \tilde K_{\chi_1, \chi_2}(m_1, m_2, f, c) = \chi_1 {\chi_2}^2 (c_0) \tilde K_{\chi_1, \chi_2}^{ (1) } \tilde K_{\chi_1, \chi_2}^{ (2) },
  \end{equation}
  with
  \begin{align*}
    \tilde K_{\chi_1, \chi_2}^{ (1) } &:= \sum_\twoln{ a_0 \bmod c_0 }{ (a_0, c_0) = 1 } \sum_\twoln{ d_1 \bmod c_0 }{d_2 \bmod c_0} e\left( \frac{ {c_2}^2 q_2 [q_1, q_2] a_0 d_1 d_2 + f a_0 + m_1 d_1 + m_2 d_2 }{ c_0 } \right), \\
    \tilde K_{\chi_1, \chi_2}^{ (2) } &:= \sum_\thrln{a_2 \bmod c_2 q_2}{ u_1 \bmod c_2 [q_1, q_2] }{u_2 \bmod c_2 q_2} \chi_1(u_1) \chi_2(a_2 u_2) e\left( \frac{ {c_0}^2 a_2 u_1 u_2 + f a_2 + m_2 u_2 }{c_2 q_2} + \frac{ m_1 u_1 }{ c_2 [q_1, q_2] } \right).
  \end{align*}
  
  In~\( \tilde K_{\chi_1, \chi_2}^{ (1) } \), we evaluate the sum over~\( d_2 \) and the whole expression immediately simplifies to
  \begin{equation}
    \tilde K_{\chi_1, \chi_2}^{ (1) } = c_0 S\big( -\overline{c_2 q_2^\circ} f, \overline{ c_2 q_2^\ast [q_1, q_2] } m_1 m_2; c_0 \big).
  \end{equation}
  In~\( \tilde K_{\chi_1, \chi_2}^{ (2) } \), we evaluate the sum over~\( u_2 \) via \cref{234} and get
  \[ \tilde K_{\chi_1, \chi_2}^{ (2) } = c_2 G(\chi_2) \sum_\thrln{a_2 \bmod c_2 q_2}{ u_1 \bmod c_2 [q_1, q_2] }{ a_2 u_1 \equiv - \overline{c_0}^2 m_2 \bmod c_2 } \chi_2(a_2) \chi_1(u_1) \overline{\chi_2}\left( \frac{ {c_0}^2 a_2 u_1 + m_2 }{c_2} \right) e\left( \frac{ f a_2 }{c_2 q_2} + \frac{ m_1 u_1 }{ c_2 [q_1, q_2] } \right). \]
  Here we write the variables \(a_2\)~and~\(u_1\) as
  \begin{alignat*}{2}
    a_2 &= a^\ast c_2 q_2^\circ + a^\circ q_2^\ast \quad &\text{with} \quad a^\circ &\bmod c_2 q_2^\circ \quad \text{and} \quad a^\ast \bmod q_2^\ast, \\
    u_1 &= u^\ast q_1^\circ q_2^\circ c_2 + v [q_1^\ast, q_2] c_2 + u^\circ [q_1, q_2^\ast] \quad &\text{with} \quad u^\circ &\bmod c_2 q_2^\circ, \quad v \bmod q_1^\circ \quad \text{and} \quad u^\ast \bmod [q_1^\ast, q_2^\ast],
  \end{alignat*}
  so that
  \begin{equation}
    \tilde K_{\chi_1, \chi_2}^{ (2) } = c_2 \overline{\chi_1^\circ}(m_1) \chi_1^\circ( c_2 [q_1^\ast, q_2] ) \chi_1^\ast( c_2 q_1^\circ q_2^\circ ) \chi_2^\circ(q_2^\ast) \overline{\chi_2^\ast}( q_1^\circ {c_0}^2 q_2^\circ ) G(\chi_1^\circ) G(\chi_2) \tilde K_{\chi_1, \chi_2}^{ (2 \text a) } \tilde K_{\chi_1, \chi_2}^{ (2 \text b) },
  \end{equation}
  with
  \begin{align*}
    \tilde K_{\chi_1, \chi_2}^{ (2 \text a) } &:= \sum_\twoln{ a^\circ, u^\circ \bmod c_2 q_2^\circ }{ a^\circ u^\circ \equiv - m_2 \bmod c_2 } \chi_2^\circ(a^\circ) \overline{\chi_2^\circ}\left( \frac{ a^\circ u^\circ + m_2 }{c_2} \right) e\bigg( \frac{ f a^\circ + m_1 \overline{ {c_0}^2 q_2^\ast [q_1, q_2^\ast] } u^\circ }{ c_2 q_2^\circ } \bigg), \\
    \tilde K_{\chi_1, \chi_2}^{ (2 \text b) } &:= \sum_\twoln{ a^\ast \bmod q_2^\ast }{ u^\ast \bmod [q_1^\ast, q_2^\ast] } \chi_1^\ast(u^\ast) \chi_2^\ast(a^\ast) \overline{\chi_2^\ast}( a^\ast u^\ast + m_2 ) e\bigg( \frac{ f \overline{ q_1^\circ (c_0 c_2 q_2^\circ)^2 } a^\ast }{q_2^\ast} \bigg) e\bigg( \frac{ m_1 u^\ast }{ [q_1^\ast, q_2^\ast] } \bigg).
  \end{align*}
  
  In the first sum~\( \tilde K_{\chi_1, \chi_2}^{ (2 \text a) } \), we make the substitution~\( u^\circ \mapsto \overline{ a^\circ } ( u^\circ - m_2 ) \), which leads to
  \begin{equation}
    \tilde K_{\chi_1, \chi_2}^{ (2 \text a) } = \chi_2^\circ(m_1) \overline{\chi_2^\circ}( {c_0}^2 q_2^\ast [q_1, q_2^\ast] ) G( \overline{\chi_2^\circ} ) S( -\overline{c_0} f, \overline{ c_0 q_2^\ast [q_1, q_2^\ast] } m_1 m_2 ; c_2 q_2^\circ ).
  \end{equation}
  In order to evaluate the second sum~\( \tilde K_{\chi_1, \chi_2}^{ (2 \text b) } \), we factorize~\(f\) as follows,
  \[ f = f^\ast f^\circ \quad \text{with} \quad f^\ast := (f, q_2^\ast) \quad \text{and} \quad f^\circ := f / f^\ast, \]
  and then express the first exponential in terms of Dirichlet characters mod~\( q_2^\ast / f^\ast \),
  \[ e\bigg( \frac{ f \overline{ q_1^\circ (c_0 c_2 q_2^\circ)^2 } a^\ast }{q_2^\ast} \bigg) = \frac1{ \phi(q_2^\ast / f^\ast) } \sum_{ \psi \bmod q_2^\ast / f^\ast } \psi( - f^\circ \overline{ q_1^\circ (c_0 c_2 q_2^\circ)^2 } a^\ast ) \overline{ G(\psi) }. \]
  This way we get
  \begin{equation} \label{4225}
    \tilde K_{\chi_1, \chi_2}^{ (2 \text b) } = \frac{ q_2^\ast [q_1^\ast, q_2^\ast]^\frac12 }{ {f^\ast}^\frac12 \phi(q_2^\ast / f^\ast) } \sum_{ \psi \bmod q_2^\ast / f^\ast } \psi(-f^\circ) \overline\psi( q_1^\circ(c_0 c_2 q_2^\circ)^2 ) \tilde E_{\chi_1, \chi_2}(m_1, m_2; \psi),
  \end{equation}
  with
  \[ \tilde E_{\chi_1, \chi_2}(m_1, m_2; \psi) := \frac{ G( \chi_1^\ast \overline{ \chi_2^\ast \psi }, m_1 ) \overline{ G(\psi) } }{ (q_2^\ast / f^\ast)^\frac12 {q_2^\ast}^\frac12 [q_1^\ast, q_2^\ast]^\frac12 } \sum_{ a \bmod q_2^\ast } \psi \chi_2^\ast(a) \overline{\chi_2^\ast}( a + m_2 ). \]
  Eventually, the lemma follows from~\eqref{4264}--\eqref{4225}.
\end{proof}

We conclude the section with the following bound for~\( E_{\chi_1, \chi_2}(m; \psi) \).

\begin{lemma}
  We have
  \[ \left| E_{\chi_1, \chi_2}(m; \psi) \right| \leq (m, q_1 q_2) \tau(m). \]
\end{lemma}
\begin{proof}
  This is a direct consequence of the bounds
  \[ \left| G( \chi_1^\ast \overline{ \chi_2^\ast \psi }, m_1 ) \right| \leq ( m_1, q_1^\ast q_2^\ast )^\frac12 [q_1^\ast, q_2^\ast]^\frac12 \quad \text{and} \quad \Bigg| \sum_{ a \bmod q_2^\ast } \psi \chi_2^\ast(a) \overline{\chi_2^\ast}( a + m_2 ) \Bigg| \leq (m_2, q_2^\ast)^\frac12 {q_2^\ast}^\frac12, \]
  see~\cite[Lemma~5.4]{MV75} and~\cite[Theorem~2.2]{OPP17}.
\end{proof}

\subsection{Technical preparations} \label{43}

Now that we have expressed the sum~\( \Sigma_{j_1, j_2}^\pm(M) \) as a sum of Kloosterman sums, the next step would be to apply the Kuznetsov formula.
However, before we can do so, some technical preparations need to be done first.

Let \( \iota_0 := 1 \)~or~\( \iota_0 := -1 \) depending on whether \(h\) is supported on the positive or negative real numbers.
Using \cref{421} we write the sum~\( \Sigma_{j_1, j_2}^\pm(M) \) as
\[ \Sigma_{j_1, j_2}^\pm(M) = \chi_2(\mp 1) \kappa_{\chi_1, \chi_2} \sum_{h^\ast \mid q_2^\ast} \sum_{n_0} \frac{ \overline{\chi_1} \chi_2(n_0) }{n_0} \Xi_{j_1, j_2}^\pm(M), \]
where
\begin{align*}
  \Xi_{j_1, j_2}^\pm(M) &:= \frac1{ \phi(h^\ast) } \sum_{ \psi \bmod h^\ast } \sum_\twoln{ h, m }{ (h, h^\ast) = 1 } \psi(\mp \iota_0 h) E_{\chi_1, \chi_2}(m; \psi) \\
     &\phantom{ := {} } \cdot \sum_{ ( c, q_1 ) = 1 } \overline{\psi^2}(c) \frac{ S( \mp \iota_0 h, \overline{ h^\ast [q_1, q_2^\ast] } m ; c q_2^\circ ) }{ c \sqrt{ h^\ast q_2^\circ [q_1, q_2] } } F_{h, m}^\pm\left( \frac{ 4\pi }c \sqrt{ \frac{ hm }{ h^\ast q_2^\circ [q_1, q_2] } } \right), \\
  \intertext{with}
  F_{h, m}^\pm(\eta) &:= \int \! B_{\chi_1, \chi_2}^\pm\left( \frac{ \eta \xi }{4 \pi} \right) U_{h, m}(\eta, \xi) e\left( \iota_0 \frac\alpha{\xi^2} \right) \, \d\xi, \\
  \intertext{and}
  U_{h, m}(\eta, \xi) &:= \iota_0 \frac{ \xi \eta }{ 2\pi } \sqrt{ \frac{ {h^\ast}^3 [q_1, q_2] }{ q_2^\ast hm } } u\left( \frac mM \right) f\left( \xi^2 h \frac{ q_2^\ast }{  h^\ast }, \iota_0 h \frac{q_2^\ast}{h^\ast} \right) \\
    &\phantom{ := {} } \cdot u_{j_1}\left( 4\pi \frac{ n_0 }{ \eta } \sqrt{ \frac{ hm }{ h^\ast q_2^\circ [q_1, q_2] } } \right) u_{j_2}\left( \frac{ \eta ( \xi^2 + \iota_0 ) }{ 4\pi n_0 } \sqrt{ \frac{ h q_2^\ast q_2 [q_1, q_2] }{h^\ast m} } \right).
\end{align*}
We also set
\[ X := \sqrt{ \frac NH }, \quad  Y := 4\pi \frac{ n_0 }{N_1} \sqrt{ \frac{HM}{ q_2 [q_1, q_2] } }, \quad Z := XY, \quad E := \frac{ h^\ast H }{ q_2^\ast }, \quad C := \frac{N_1}{n_0}, \quad F_0 := \frac{ h^\ast n_0 N_2}{ {q_2}^\frac12 H }. \]
With this notation, the different variables are supported in the intervals
\[ \xi \in [ X/3, 3X ], \quad \eta \in [ Y/120, 120Y ], \quad |h| \in [ E/4, 2E ], \quad m \in [ M/4, 2 M ], \quad c \in [ C/9, 9C ], \]
provided that~\(N\) is sufficiently large.
Also note that the variable~\(n_0\) is bounded by~\( n_0 \ll N_1 \).

We next want to show that the sums~\( \Sigma_{j_1, j_2}^\pm(M) \) become negligibly small when \(M\)~is in certain ranges.
Let~\( \eps_0 > 0 \) be an arbitrarily small but fixed constant, and set
\[ M_0^- := N^{\eps_0} \frac{ q_2 [q_1, q_2] }{16\pi^2 N} \left( \frac{N_1}{n_0} \right)^2 \qquad \text{and} \qquad M_0^+ := \frac{ q_2 [q_1, q_2] }{16\pi^2 N} \left( \frac{N_1}{n_0} \right)^2 \left( \frac{\alpha H}N \right)^2. \]
If \(M\) satisfies the bound~\( M > M_0^- \), which is equivalent to saying that~\( Z > N^{\eps_0 / 2} \), then by well-known properties of the \(K_0\)-Bessel function (see e.g.~\cite[(B.36)]{Iwa02}), we have
\[ F_{h, m}^-(\eta) \ll F_0 \exp\bigg( {-} \frac{ N^{ \eps_0 / 4 } M^{1/2} }{10} \bigg). \]
Hence the contribution coming from the sums~\( \Sigma_{j_1, j_2}^-(M) \) for such large~\(M\) is negligible.
By consequence, when looking at~\( \Sigma_{j_1, j_2}^-(M) \) we can therefore safely assume that~\( M \ll M_0^- \).

Similarly, if~\( M > M_0^- \), then we can express~\( F_{h, m}^+(\eta) \) by \cref{233} as
\[ F_{h, m}^+(\eta) = \int \! \left( W_{\chi_1, \chi_2}\left( \frac{ \xi \eta }{ 4 \pi } \right) U_{h, m}(\eta, \xi) e\bigg( \frac{ \iota_0 \alpha }{\xi^2} + \frac{ \xi \eta }{ 2 \pi } \bigg) + \overline{ W_{\chi_1, \chi_2}\left( \frac{ \xi \eta }{ 4 \pi } \right) } U_{h, m}(\eta, \xi) e\bigg( \frac{ \iota_0 \alpha }{\xi^2} - \frac{ \xi \eta }{ 2 \pi } \bigg) \right) \d\xi. \]
If we now make the additional assumption that
\begin{equation} \label{4350}
  \alpha X^{-3} \geq 10^6 Y \quad \text{or} \quad \alpha X^{-3} \leq 10^{-6} Y,
\end{equation}
then
\[ \left| \frac\partial{\partial \xi} \bigg( \iota_0 \frac{ \alpha }{\xi^2} \pm \frac{ \xi \eta }{ 2 \pi } \bigg) \right| = \left| -2 \iota_0 \frac\alpha{\xi^3} \pm \frac\eta{ 4\pi } \right| \geq \frac Y{10^4}, \]
so that by integrating by parts over~\(\xi\) repeatedly it follows that, for any~\( \nu \),
\[ F_{h, m}^+(\eta) \ll F_0 Z^{-\nu} \ll F_0 N^{ -\eps_0 \nu / 4 } M^{-\nu / 2}. \]
Hence we see that the contribution coming from those sums~\( \Sigma_{j_1, j_2}^+(M) \) where \(M\) satisfies both \( M > M_0^- \)~and~\eqref{4350} is negligible.
When looking at~\( \Sigma_{j_1, j_2}^+(M) \) we can therefore assume that \(M\) is either bounded by~\( M \ll M_0^- \), or that it satisfies the two conditions~\( M \gg M_0^- \) and~\( Y \asymp \alpha X^{-3} \).
Note that the latter condition~\( Y \asymp \alpha X^{-3} \) is equivalent to saying that~\( M \asymp M_0^+ \).

Due to technical reasons it is necessairy to separate the variables \(h\) and~\(m\) via Fourier inversion.
To this end, we define
\[ G_{\rho, \lambda}^\pm(\eta) := \frac1{ G_{\rho, \lambda}^0 } \iint \! F_{h, m}^\pm(\eta) e( -\rho h - \lambda m ) \, \d h \d m \qquad \text{with} \qquad G_{\rho, \lambda}^0 := \frac{ E M}{ (1 + \rho^2 E^2) (1 + \lambda^2 M^2) }, \]
so that
\begin{multline*}
  \Xi_{j_1, j_2}^\pm(M) = \iint \! G_{\rho, \lambda}^0 \frac1{ \phi(h^\ast) } \sum_{ \psi \bmod h^\ast } \sum_\twoln{ h, m }{ (h, h^\ast) = 1 } \psi(\mp \iota_0 h) e(\rho h) E_{\chi_1, \chi_2}(m; \psi) e(\lambda m) \\
     \cdot \sum_{ ( c, q_1 ) = 1 } \overline{\psi^2}(c) \frac{ S( \mp \iota_0 h, \overline{ h^\ast [q_1, q_2^\ast] } m ; c q_2^\circ ) }{ c \sqrt{ h^\ast q_2^\circ [q_1, q_2] } } G_{\rho, \lambda}^\pm\left( \frac{ 4\pi }c \sqrt{ \frac{ hm }{ h^\ast q_2^\circ [q_1, q_2] } } \right) \, \d\rho \, \d\lambda.
\end{multline*}

Last but not least, we need estimates for the integral transforms of~\( G_{\rho, \lambda}^\pm \) as defined in~\eqref{3393}--\eqref{3382}.
Note that in our case it suffices to consider the integral transforms associated to even characters.

We start with the case~\( M \leq M_0^- \).

\begin{lemma} \label{431}
  Assume that~\( M \leq M_0^- \).
  Then we have, for any~\( \nu \geq 0 \),
  \begin{alignat}{2}
    \tilde G_{\rho, \lambda}^\pm(\i t), \check G_{\rho, \lambda}^\pm(\i t) &\ll \frac{ F_0 }{ Y^{2t} } \quad &&\text{for} \quad 0 \leq t \leq 1/4, \label{4304} \\
    \tilde G_{\rho, \lambda}^\pm(t), \check G_{\rho, \lambda}^\pm(t), \dot G_{\rho, \lambda}^\pm(t) &\ll N^\eps F_0 \left( \frac{N^\eps}t \right)^\nu \quad &&\text{for} \quad t > 0. \label{4310}
  \end{alignat}
\end{lemma}
\begin{proof}
  It is clearly sufficient to look directly at the function~\( F_{h, m}^\pm(\eta) \) and its first two partial derivatives in~\( h \) and~\(m\).
  Noting that~\( Y \ll 1 \), and that
  \[ \supp F_{h, m}^\pm \subset [ Y/120, 120 Y ] \qquad \text{and} \qquad F_{h, m}^{ \pm (\nu) }(\eta) \ll N^\eps F_0 ( N^\eps / Y )^\nu \quad \text{for} \quad \nu \geq 0, \]
  we apply~\cite[Lemma~2.1]{BHM07} on~\( F_{h, m}^\pm(\eta) \) and its partial derivatives in~\(h\) and~\(m\), and \eqref{4304}~and~\eqref{4310} eventually follow.
\end{proof}

Next, we consider the case~\( M > M_0^- \), which requires a more delicate analysis.
As argued above, this only involves the function~\( F_{h, m}^+(\eta) \), and we can assume that~\( M \asymp M_0^+ \).
Remember that now we also have~\( Z > N^{\eps_0 / 2} \).

\begin{lemma} \label{432}
  Assume that~\( M > M_0^- \) and~\( M \asymp M_0^+ \).
  Then we have, for any~\( \nu \geq 0 \),
  \begin{alignat}{2}
    \tilde G_{\rho, \lambda}^+(\i t), \check G_{\rho, \lambda}^+(\i t) &\ll \frac{F_0}{N^\nu} \quad &&\text{for} \quad 0 \leq t \leq 1/4, \label{4393} \\
    \tilde G_{\rho, \lambda}^+(t), \check G_{\rho, \lambda}^+(t), \dot G_{\rho, \lambda}^+(t) &\ll N^\eps \frac{ F_0 }{ Z^2 } \left( \frac Zt \right)^\nu \quad &&\text{for} \quad t > 0. \label{4349}
  \end{alignat}
\end{lemma}
\begin{proof}
  As before, it is enough to consider the function~\( F_{h, m}^+(\eta) \) and its first two partial derivatives in~\(h\) and~\(m\).
  We will restrict our attention here to~\( F_{h, m}^+(\eta) \) itself, since the analogous bounds for its derivatives can be derived similarly.
  Moreover, we will make the additional assumption~\( \iota_0 = -1 \), since the other case~\( \iota_0 = 1 \) can be treated almost identically.
  
  We start by using \cref{233} to write~\( F_{h, m}^+(\eta) \) as
  \[ F_{h, m}^+(\eta) = \Phi^+(\eta) + \Phi^-(\eta) \qquad \text{with} \qquad \Phi^\pm(\eta) := \int \! V_\xi^\pm(\eta) e\bigg( {-}\frac\alpha{\xi^2} \pm \frac{ \xi \eta }{ 2 \pi } \bigg) \, \d\xi, \]
  where \( V_\xi^+(\eta) \)~and~\( V_\xi^-(\eta) \) are given by
  \[ V_\xi^+(\eta) := W_{\chi_1, \chi_2}\left( \frac{ \xi \eta }{ 4 \pi } \right) U_{h, m}(\eta, \xi) \qquad \text{and} \qquad V_\xi^-(\eta) := \overline{ W_{\chi_1, \chi_2}\left( \frac{ \xi \eta }{ 4 \pi } \right) } U_{h, m}(\eta, \xi). \]
  Note that
  \[ \supp V_\xi^\pm \subset [ Y/120, 120 Y ] \qquad \text{and} \qquad V_\xi^{ \pm (\nu) }(\eta) \ll F_0 X^{-1} Z^{-\frac12} Y^{-\nu} \quad \text{for} \quad \nu \geq 0. \]
  Furthermore, the assumption~\eqref{4021} ensures that~\( Y \ll N^{-\eps} \).
  Hence we can apply~\cite[Lemma~2.6]{Top16} on the function~\( V_\xi^\pm(\eta) e( \pm (2\pi)^{-1} \xi \eta ) \), and get
  \begin{alignat*}{2}
    \tilde \Phi^\pm(\i t), \check \Phi^\pm(\i t) &\ll F_0 N^{-\nu} \quad &&\text{for} \quad 0 \leq t \leq 1/4, \\
    \tilde \Phi^\pm(t), \check \Phi^\pm(t), \dot \Phi^\pm(t) &\ll N^\eps F_0 Z^{-\frac32} (Z/t)^\nu \quad &&\text{for} \quad t > 0.
  \end{alignat*}
  This proves the first bound~\eqref{4393}, but also the second bound~\eqref{4349} in the range~\( t \gg N^\eps Z \).
  
  It thus remains to estimate the integral transforms of~\( \Phi^\pm(\eta) \) for~\( t \ll N^\eps Z \).
  In~\( \Phi^+(\eta) \), we integrate by parts over~\(\xi\) once and then apply one more time~\cite[Lemma~2.6]{Top16}.
  This gives
  \begin{equation} \label{4338}
    \tilde \Phi^+(t), \check \Phi^+(t), \dot \Phi^+(t) \ll N^\eps F_0 Z^{-5/2} \quad \text{for} \quad t > 0,
  \end{equation}
  which is sufficiently small.
  Unfortunately, we cannot repeat this procedure to get bounds for the integral transforms of~\( \Phi^-(\eta) \), since the argument of the exponential in~\( \Phi^-(\eta) \) may vanish.
  Instead, we will estimate the integral transforms manually via a stationary phase argument, and show that
  \begin{equation} \label{4329}
    \tilde \Phi^-(t), \check \Phi^-(t), \dot \Phi^-(t) \ll N^\eps F_0 Z^{-2} \quad \text{for} \quad t > 0.
  \end{equation}
  
  We begin with~\( \tilde \Phi^-(t) \).
  It will be convenient to have a smooth bump function of a certain shape at hand.
  To this end, we let~\( v_0 : \RR \to [0, 1] \) be a smooth and compactly supported function such that
  \[ v_0(\xi) = 1 \quad \text{for} \quad |\xi| \leq 1 \qquad \text{and} \qquad v_0(\xi) = 0 \quad \text{for} \quad |\xi| \geq 2, \]
  and furthermore define~\( v_1(\xi) := 1 - v_0(\xi) \).
  
  Assume first that~\( t \ll N^\eps \).
  Using~\cite[8.411.11]{GR07}, we write~\( \tilde \Phi^-(t) = I^+ + I^- \) with
  \[ I^\pm = -\iint \! \int_1^\infty \cos(2t \arcosh\zeta) \frac{ V_\xi^-(\eta) }{ { \eta \sqrt{\zeta^2 - 1} } } e\bigg( {-}\frac\alpha{\xi^2} - \frac{ \xi \eta }{ 2 \pi } \pm\frac{ \eta \zeta }{2\pi} \bigg) \, \d\zeta \d\eta \d\xi. \]
  Integrating by parts over~\(\eta\) repeatedly shows that the integral~\( I^- \) is arbitrarily small.
  We split the other integral into two parts~\( I^+ = I_0^+ + I_1^+ \) with
  \[ I_j^+ = -\iint \! \int_1^\infty \cos(2t \arcosh\zeta) v_j\left( \frac{ \xi - \zeta }{X/12} \right) \frac{ V_\xi^-(\eta) }{ { \eta \sqrt{\zeta^2 - 1} } } e\bigg( {-}\frac\alpha{\xi^2} - \frac{ \xi \eta }{ 2 \pi } \pm\frac{ \eta \zeta }{2\pi} \bigg) \, \d\zeta \d\eta \d\xi. \]
  In~\( I_1^+ \), we integrate by parts over~\( \eta \) repeatedly to see that its size is negligible.
  In~\( I_0^+ \), we observe that~\( \zeta \asymp X \) and integrate by parts over~\( \zeta \) repeatedly to see that this integral is also negligibly small.
  Hence \eqref{4329}~is certainly true.
  
  Now assume~\( N^\eps \ll t \ll N^\eps Z \).
  Since~\( Y \ll N^{-\eps} \), we can use~\cite[(B.28)]{Iwa02} to express the Bessel function~\( J_{2\i t}(\eta) \) inside the integral transform~\eqref{3393} as
  \[ J_{2 \i t}(\eta) = \Gamma(2\i t + 1)^{-1} \eta^{2\i t} W_t(\eta), \]
  where \( W_t(\eta) \) is a certain complex-valued function which, uniformly in~\(t\), satisfies the bounds
  \[ W_t^{ (\nu) }(\eta) \ll \eta^{-\nu} \quad \text{for} \quad \nu \geq 0. \]
  It follows that
  \[ \tilde \Phi^-(t) \ll t^{-\frac12} \left( | L^+ | + | L^- | \right), \]
  with
  \[ L^\pm := \iint \! e\left( A_0^\pm(\xi, \eta) \right) V_\xi^-(\eta) W_{\pm t}(\eta) \, \frac{\d\xi d\eta}\eta \qquad \text{and} \qquad A_0^\pm(\xi, \eta) := \pm\frac{ t \log\eta }\pi -\frac\alpha{\xi^2} - \frac{\xi \eta}{2\pi}. \]
  Integrating by parts over~\( \eta \) repeatedly shows that~\( L^- \) is negligibly small.
  By the same reasoning we see that~\( L^+ \) too is negligible, unless \(t\)~is of the size~\( t \asymp Z \) which we will henceforth assume.
  
  We split the double integral~\( L^+ \) via the weight functions \(v_0\)~and~\(v_1\) defined above into four parts~\( L^+ = L_{0, 0}^+ + L_{1, 0}^+ + L_{0, 1}^+ + L_{1, 1}^+ \), where
  \[ L_{j_1, j_2}^+ = \iint \! e\left( A_0^+(\xi, \eta) \right) v_{j_1}\left( \frac{ A_1(\xi, \eta) }{ N^\eps (X/Y)^\frac12 } \right) v_{j_2}\left( \frac{ A_2(\xi, \eta) }{ N^\eps (Y/X)^\frac12 } \right) V_\xi^-(\eta) W_t(\eta) \, \frac{\d\xi d\eta}\eta, \]
  with
  \[ A_1(\xi, \eta) := \frac\partial{\partial\eta} A_0^+(\xi, \eta) = \frac t{\pi\eta} - \frac\xi{2\pi} \qquad \text{and} \qquad A_2(\xi, \eta) := \frac\partial{\partial\xi} A_0^+(\xi, \eta) = \frac{2\alpha}{\xi^3} - \frac\eta{2\pi}. \]
  Integration by parts, either over~\(\xi\) or over~\(\eta\), shows once more that~\( L_{1, 0}^+ \), \( L_{0, 1}^+ \) and~\( L_{1, 1}^+ \) are all of negligible size, so that we can focus on the remaining integral~\( L_{0, 0}^+ \) .
  
  Here we make the substitution
  \[ (\xi, \eta) = \psi(\zeta_1, \zeta_2) \qquad \text{with} \qquad \psi(\zeta_1, \zeta_2) := \left( \alpha_0 + \zeta_1, 2t ( \alpha_0 + \zeta_1 + 2\pi \zeta_2 )^{-1} \right), \]
  where we have set~\( \alpha_0 := (2\pi\alpha)^{1/2} t^{-1/2} \).
  Note that~\( \alpha_0 \asymp X \) and~\( (\alpha_0 + \zeta_1) \asymp X \).
  This gives
  \[ L_{0, 0}^+ \ll \frac{ F_0 }{ Z^\frac32 } \frac YX \iint \! v_0\left( \frac{ A_1( \psi(\zeta_1, \zeta_2) ) }{ N^\eps (X/Y)^\frac12 } \right) v_0\left( \frac{ A_2( \psi(\zeta_1, \zeta_2) ) }{ N^\eps (Y/X)^\frac12 } \right) \, \d\zeta_1 \d\zeta_2. \]
  As we will show below, the two integration variables \(\zeta_1\)~and~\(\zeta_2\) are both supported in~\( \zeta_1, \zeta_2 \ll N^\eps (X/Y)^{1/2} \).
  As a consequence, it follows that~\( L_{0, 0}^+ \ll N^\eps F_0 Z^{-3/2} \), which in turn directly leads to~\eqref{4329}.
  
  Concerning~\( A_1( \psi(\zeta_1, \zeta_2) ) \), we have
  \[ A_1( \psi(\zeta_1, \zeta_2)) = \zeta_2, \]
  which immediately confirms that the integration variable~\( \zeta_2 \) is bounded by~\( N^\eps (X/Y)^{1/2} \).
  Concerning~\( A_2( \psi(\zeta_1, \zeta_2) ) \), a quick calculation shows that
  \[ A_2( \psi(\zeta_1, \zeta_2)) = -\frac{ t \zeta_1 ( 2\alpha_0 + \zeta_1 ) }{ \pi ( \alpha_0 + \zeta_1 )^3 } + \frac{ 2 \zeta_2 t }{ ( \alpha_0 + \zeta_1 ) ( \alpha_0 + \zeta_1 + 2\pi \zeta_2 ) }. \]
  Since the second summand on the right hand side is bounded by~\( N^\eps (Y/X)^{1/2} \), we see that for the expression~\( A_2( \psi(\zeta_1, \zeta_2) ) \) to be bounded by~\( N^\eps (Y/X)^{1/2} \), we must have
  \[ \frac{ t \zeta_1 ( \zeta_1 + 2\alpha_0 ) }{ \pi ( \zeta_1 + \alpha_0 )^3 } \ll N^\eps \frac{ Y^\frac12 }{ X^\frac12 }, \]
  which is possible only if~\( \zeta_1 \ll N^\eps (X/Y)^{1/2} \).
  
  The integral transform~\( \check \Phi^-(t) \) can be treated similarly by using suitable integral representations for the Bessel function~\( K_{2\i t}(\eta) \), for example~\cite[8.432.4]{GR07} and~\cite[(B.32) and~(B.34)]{Iwa02}.
  Finally, in order to bound the integral transform~\( \dot \Phi^-(t) \), we express the Bessel function~\( J_{k - 1}(\eta) \) via the integral representation~\cite[8.411.1]{GR07} and then integrate by parts repeatedly over~\( \eta \), which already gives the desired bound.
  This finishes the proof of \cref{432}.
\end{proof}

\subsection{Use of the Kuznetsov formula} \label{44}

We are finally ready to apply the Kuznetsov formula on the sums~\( \Xi_{j_1, j_2}^\pm(M) \).
Specifically, we will use \cref{331} in the form~\eqref{3343} with parameters
\[ \tilde \psi := \psi^2, \quad \tilde q_0 := h^\ast, \quad \tilde r := h^\ast [q_1, q_2^\ast], \quad \tilde s := q_2^\circ, \quad \tilde q := h^\ast [q_1, q_2]. \]
We will give the details only for~\( \Xi_{j_1, j_2}^+(M) \) and assume that~\( \iota_0 = -1 \), since the other sums and cases can all be treated in the same manner.

Using the Kuznetsov formula as described above leads to
\begin{align*}
  \Xi^+_{j_1, j_2}(M) = \iint \! G_{\rho, \lambda}^0 \left( \Xi_1 + \Xi_2 + \Xi_3 \right) \, \d\rho \d\lambda,
\end{align*}
where \(\Xi_1\), \(\Xi_2\) and~\( \Xi_3 \) are given by
\begin{align*}
  \Xi_1 &:= \frac1{ \phi(h^\ast) } \sum_{ \psi \bmod h^\ast } \sum_{ j \geq 0 } \tilde G_{\rho, \lambda}^+\big( t_j^{\psi^2} \big) \overline{ \Sigma_{1 \text a}^\psi (j) } \Sigma_{ 1 \text b }^\psi (j), \\
  \Xi_2 &:= \frac1{ \phi(h^\ast) } \sum_{ \psi \bmod h^\ast } \sum_{ \mf c \text{ sing.} } \frac1{4\pi} \int_{-\infty}^\infty \! \tilde G_{\rho, \lambda}^+(t) \overline{ \Sigma_{2 \text a}^\psi (\mf c, t) } \Sigma_{2 \text b}^\psi (\mf c, t) \, \d t, \\
  \Xi_3 &:= \frac1{ \phi(h^\ast) } \sum_{ \psi \bmod h^\ast } \sum_\twoln{ k \geq 2, \,\, k \equiv 0 \bmod 2 }{ 1 \leq j \leq \theta_k( h^\ast [q_1, q_2], \psi^2 ) } \dot G_{\rho, \lambda}^+(k) \overline{ \Sigma_{3 \text a}^\psi (j, k) } \Sigma_{3 \text b}^\psi (j, k),
\end{align*}
with
\begin{alignat*}{3}
  \Sigma_{1 \text a}^\psi (j) &:= \sum_{ E/4 < h \leq 2E } A_1^\psi(h) \rho_j^{\psi^2}( h, \infty ), \qquad &\Sigma_{2 \text a}^\psi (j) &:= \sum_{ M/4 < m \leq 2 M } A_2^\psi(m) \rho_j^{\psi^2}\left( m, 1 / q_2^\circ \right), \\
  \Sigma_{1 \text b}^\psi (\mf c, t) &:= \sum_{ E/4 < h \leq 2E } A_1^\psi(h) \phi_{ \mf c, t }^{\psi^2}( h, \infty ), \qquad & \Sigma_{2 \text b}^\psi (\mf c, t) &:= \sum_{ M / 4 < m \leq 2 M } A_2^\psi(m) \phi_{ \mf c, t }^{\psi^2}\left( m, 1 / q_2^\circ \right), \\
  \Sigma_{1 \text c}^\psi (j, k) &:= \sum_{ E/4 < h \leq 2E } A_1^\psi(h) \lambda_{j, k}^{\psi^2}( h, \infty ), \qquad &\Sigma_{2 \text c}^\psi (j, k) &:= \sum_{ M / 4 < m \leq 2 M } A_2^\psi(m) \lambda_{j, k}^{\psi^2}\left( m, 1 / q_2^\circ \right),
\end{alignat*}
and
\[ A_1^\psi(h) := \overline\psi(h) e(-\rho h), \qquad A_2^\psi(m) := E_{\chi_1, \chi_2}(m; \psi) e(\lambda m) e\left( -\frac{ \overline{q_2^\circ} m }{ h^\ast [q_1, q_2^\ast] } \right). \]

We first consider the case~\( M \ll M_0^- \).
We split the sum~\( \Xi_1 \) into three parts,
\[ \Xi_1 = \frac1{ \phi(h^\ast) } \sum_\twoln{\psi \bmod h^\ast}{ t_j^{\psi^2} \leq N^\eps} (\ldots) + \frac1{ \phi(h^\ast) } \sum_\twoln{\psi \bmod h^\ast}{ t_j^{\psi^2} > N^\eps } (\ldots) + \frac1{ \phi(h^\ast) } \sum_\twoln{\psi \bmod h^\ast}{ \text{\( t_j^{\psi^2} \)exc.} } (\ldots) =: \Xi_{1 \text a} + \Xi_{1 \text b} + \Xi_{1 \text c}. \]
By \cref{431} it is clear that the contribution coming from~\( \Xi_{1 \text b} \) is negligible.
Concerning~\( \Xi_{1 \text a} \), we make use of the bound~\eqref{4310} and apply Cauchy-Schwarz, so that
\[ \Xi_{1 \text a} \ll N^\eps F_0 \Bigg( \frac1{ \phi(h^\ast) } \sum_{ \psi \bmod h^\ast } \sum_{ t_j^{\psi^2} \leq N^\eps } \left| \Sigma_{1 \text a}^\psi(j) \right|^2 \Bigg)^\frac12 \Bigg( \frac1{ \phi(h^\ast) } \sum_{ \psi \bmod h^\ast } \sum_{ t_j^{\psi^2} \leq N^\eps } \left| \Sigma_{1 \text b}^\psi(j) \right|^2 \Bigg)^\frac12. \]
Applying \cref{341} on the sums inside the two factors then leads to
\[ \Xi_{1 \text a} \ll N^\eps F_0 \left( 1 + \frac{ E^\frac12 }{ ( h^\ast [q_1, q_2] )^\frac12 } \right) \left( 1 + \frac{M^\frac12}{ ( h^\ast [q_1, q_2] )^\frac12 } \right) E^\frac12 M^\frac12 \ll (q_2^\ast q_2 [q_1, q_2])^\frac12 N^{\frac12 + \eps}. \]
Note that we have made here implicitly use of the fact that, for a given Dirichlet character~\( \tilde \psi \)~mod~\(h^\ast\), there are at most \( {h^\ast}^\eps \)~many Dirichlet characteres~\( \psi \)~mod~\(h^\ast\) such that~\( \psi^2 = \tilde \psi \).

For~\( \Xi_{1 \text c} \) the same approach leads, for~\( H \gg h^\ast q_2^\ast [q_1, q_2]^2 \), to
\[ \Xi_{1 \text c} \ll (q_2^\ast q_2 [q_1, q_2])^\frac12 \left( \frac N{ ( q_2^\ast [q_1, q_2] )^2 } \right)^\theta N^{\frac12 + \eps}. \]
For~\( H \ll h^\ast q_2^\ast [q_1, q_2]^2 \) we make use of \cref{342} instead of \cref{341} to estimate the sum over~\(h\), which gives
\begin{align*}
  \Xi_{1 \text c} &\ll \frac{ N^\eps F_0 E^\theta }{ (Y h^\ast [q_1, q_2])^{2\theta} } \Bigg( \frac1{ \phi(h^\ast) } \sum_{ \psi \bmod h^\ast } \sum_{ \text{\( t_j^{\psi^2} \)exc.} } \left( \frac{ h^\ast [q_1, q_2] }{ E^\frac12 } \right)^{ 4\i t_j^{\psi^2} } \left| \Sigma_{1 \text a}^\psi(j) \right| \Bigg)^\frac12 \\
    &\phantom{ \ll \frac{ N^\eps F_0 E^\theta }{ (Y h^\ast [q_1, q_2])^{2\theta} } } \!\!\! \cdot \Bigg( \frac1{ \phi(h^\ast) } \sum_{ \psi \bmod h^\ast } \sum_{ \text{\( t_j^{\psi^2} \)exc.} } \left| \Sigma_{1 \text b}^\psi(j) \right|^2 \Bigg)^\frac12 \\
    &\ll ( q_2^\ast q_2 [q_1, q_2] )^\frac12 \left( \frac N{ ( q_2^\ast [q_1, q_2] )^2 } \right)^\theta N^{\frac12 + \eps}.
\end{align*}
The two other sums \( \Xi_2 \)~and~\( \Xi_3 \) can be estimated similarly, except that there are no exceptional eigenvalues to be taken care of.
The upper bound we get for these two sums is the same as the one for~\( \Xi_{1 \text a} \).

Next we look at the case where \( M \gg M_0^- \) and~\( M \asymp M_0^+ \).
As before we split the sum~\( \Xi_1 \) into three parts,
\[ \Xi_1 = \frac1{ \phi(h^\ast) } \sum_\twoln{\psi \bmod h^\ast}{ t_j^{\psi^2} \leq N^\eps Z } (\ldots) + \frac1{ \phi(h^\ast) } \sum_\twoln{\psi \bmod h^\ast}{ t_j^{\psi^2} > N^\eps Z } (\ldots) + \frac1{ \phi(h^\ast) } \sum_\twoln{\psi \bmod h^\ast}{ \text{\( t_j^{\psi^2} \)exc.} } (\ldots) =: \Xi_{1 \text a} + \Xi_{1 \text b} + \Xi_{1 \text c}. \]
By \cref{432} we see that the contribution coming from both the terms~\( \Xi_{1 \text b} \) and~\( \Xi_{1 \text c} \) is negligible.
For~\( \Xi_{1 \text a} \) we get in the same way as above, using~\eqref{4349}, Cauchy-Schwarz and \cref{341},
\begin{align*}
  \Xi_{1 \text a} &\ll \frac{ N^\eps F_0 }{Z^2} \left( Z + \frac{ E^\frac12 }{ ( h^\ast [q_1, q_2] )^\frac12 } \right) \left( Z + \frac{M^\frac12}{ ( h^\ast [q_1, q_2] )^\frac12 } \right) E^\frac12 M^\frac12 \\
    &\ll N^\eps (q_2^\ast q_2 [q_1, q_2])^\frac12 N^{\frac12 + \eps} \left( 1 + \alpha \frac{ H^\frac12 }N \right).
\end{align*}
The same bound also holds for~\( \Xi_2 \) and~\( \Xi_3 \), as can be deduced analogously.

Putting everything together we arrive at
\[ \Xi_{j_1, j_2}^-(M) \ll ( q_2^\ast q_2 [q_1, q_2] )^\frac12 N^{\frac12 + \eps} \left( 1 + \alpha \frac{ H^\frac12 }N + \frac{ N^\theta }{ ( q_2^\ast [q_1, q_2] )^{2\theta} } \right). \]
This eventually leads to the error term stated in \cref{401}

\subsection{The main term} \label{45}

It remains to evaluate the main term, which is formed by summing over all the terms~\eqref{4164}, and which takes the following form,
\[ M := \frac12 \sum_h \frac1h \int \! \sum_{ c = 1 }^\infty \frac{ \overline{\chi_1}(c) A_2(c) B_2(c) + \overline{\chi_2}(c) A_1(c) B_1(c) }{c^2} f( \xi, h ) e\left( \alpha \frac h\xi \right) \, \d\xi, \]
with
\begin{align*}
  A_i(c) &:= \frac1{ {q_i}^2 } \sum_\twoln{a \bmod c q_i}{ (a, c q_i) = 1 } \chi_i(a) e\left( \frac{ ha }{c q_i} \right) \overline{ G( \chi_i ) } \Pi_{\chi_1, \chi_2}( \log\xi; c q_i, a ), \\
  B_1(c) &:= \sum_n \frac{ \chi_1 \overline{\chi_2}(n) }n ( 1 + u_0(cn) ) \left( 1 - u_0\left( \frac{ \xi + h }{cn} \right) \right), \qquad B_2(c) := \overline{ B_1 (c) }.
\end{align*}

In the case~\( \chi_1 = \chi_2 \), the expression~\( A_i(c) \) simplifies to
\begin{align*}
  A_i(c) &= {q_1}^{-1} \chi_1(c) r_{c q_1}(h) \Delta_{z_1} \xi^{z_1} Z_{q_1}(2 z_1) c^{-2 z_1},
  \intertext{while \( B_i(c) \)~can be evaluated via a standard counter integration argument, leading to}
  B_i(c) &= \Delta_{z_2} ( \xi + h )^{z_2} Z_{q_1}(2 z_2) c^{-2 z_2} + \LO{ c^{1 - \eps} N^{-\frac12 + \eps} }.
\end{align*}
Put together this immediately leads to the expression stated in~\eqref{4065}.
The other case~\( \chi_1 \neq \chi_2 \) can be handled similarly.

\section{Proof of Theorems~\ref{101}--\ref{106}} \label{5}

In this section, we want to prove our main results, Theorems~\ref{101}--\ref{106}.
The general outline of the proof follows the approach described in~\cite[Chapter~4]{Ivi91}.

As before we assume \( \chi_1 \)~mod~\(q_1\) and \( \chi_2 \)~mod~\(q_2\) to be primitive Dirichlet characters.
Let
\[ q_1^\ast := \left( q_1, {q_2}^\infty \right), \quad q_2^\ast := \left( q_2, {q_1}^\infty \right) \quad \text{and} \quad q_0 := \sqrt{q_1 q_2}. \]
Instead of looking directly at \eqref{1019}~and~\eqref{1045}, it will be advantageous to look at their smooth analogues.
Hence, let~\( \delta > 0 \) be a fixed constant, let \( T_0 \)~and~\( \Omega \) be positive real numbers such that
\[ q_0 \max\{ q_1, q_2 \} \leq T_0^{1 - \delta} \qquad \text{and} \qquad {q_0}^\frac13 {T_0}^{-\frac13 + \delta} \leq \Omega \leq 1, \]
and let \( w : (0, \infty) \to [0, \infty) \) be a smooth weight function, which is compactly supported in
\[ \supp w \subset [ T_0 / 4, 2 T_0 ], \]
and whose derivatives satisfy the bounds
\begin{equation} \label{5005}
  w^{ (\nu) }(t) \ll (\Omega T_0)^{ -\nu } \quad \text{for} \quad \nu \geq 0 \qquad \text{and} \qquad \int \big| w^{ (\nu) }(t) \big| \, \d t \ll (\Omega T_0)^{1 - \nu} \quad \text{for} \quad \nu \geq 1.
\end{equation}
Our principal object of study will then be the smoothed moment
\[ I_{\chi_1, \chi_2}(w) := \int \left| L_{\chi_1, \chi_2}\left( \tfrac12 + \i \tfrac{2\pi t}{q_0} \right) \right|^2 w(t) \, \d t. \]
Compared with the original expressions~\eqref{1019} and~\eqref{1045}, we use a different normalization in~\(t\) here, as this will lead to simpler formulae during the proof.

Our aim is to prove the following asymptotic formula.

\begin{proposition} \label{501}
  Let~\( \delta, \eps > 0 \).
  Then
  \[ I_{\chi_1, \chi_2}(w) = \int \! P_{\chi_1, \chi_2}\left( \log\left( \tfrac{2\pi t}{q_0} \right) \right) w(t) \, \d t + \LO{ {T_0}^\eps E_{\chi_1, \chi_2}( T_0, \Omega) }, \]
  where \( P_{\chi_1, \chi_2} \) is a polynomial of degree at most~\(4\) whose coefficients depend only on \(\chi_1\)~and~\(\chi_2\), where \( E_{\chi_1, \chi_2}( T_0, \Omega) \) is the quantity defined as
  \begin{equation} \label{5057}
    E_{\chi_1, \chi_2}( T_0, \Omega) := \left( {q_0}^\frac12 + \frac1{ \Omega^\frac12 } \right) \frac{ ( q_1^\ast q_1 + q_2^\ast q_2 )^\frac12 }{ (q_1, q_2)^\frac12 } {q_0}^\frac32 {T_0}^\frac12 + \left( \frac{ q_1^\ast q_1 }{ {q_1^\ast}^{4\theta} } + \frac{ q_2^\ast q_2 }{ {q_2^\ast}^{4\theta} } \right)^\frac12 \frac{ {q_0}^{ 2 - 4\theta } }{ (q_1, q_2)^{\frac12 - 2\theta} } {T_0}^{\frac12 + \theta},
  \end{equation}
  and where the error depends only on~\(\delta\), \( \eps \) and the implicit constants in~\eqref{5005}.
\end{proposition}

The polynomial~\( P_{\chi_1, \chi_2} \) which appears in the main term is the same polynomial as in Theorems~\ref{101}--\ref{106} (we set~\( P_\chi := P_{\chi, \chi} \)).
We will evaluate it explicitly at the end in \cref{55}.

Applying \cref{501} with~\( T_0 = (2\pi)^{-1} q_0 T \) and~\( \Omega = 1 \) immediately gives Theorems~\ref{102}, \ref{104} and~\ref{106}.
In order to prove the other results, we again set~\( T_0 = (2\pi)^{-1} q_0 T \), and then choose smooth and compactly supported weight functions~\( w^\pm : (0, \infty) \to [0, 1] \) of the following form,
\begin{gather*}
  w^-(t) = 1 \quad \text{for} \quad t \in [ (1 + \Omega) T_0 / 2, (1 - \Omega) T_0 ] \qquad \text{and} \qquad w^-(t) = 0 \quad \text{for} \quad t \not\in [T_0/2, T_0], \\
  w^+(t) = 1 \quad \text{for} \quad t \in [T_0/2, T_0] \qquad \text{and} \qquad w^+(t) = 0 \quad \text{for} \quad t \not\in [ (1 - \Omega) T_0 / 2, (1 + \Omega) T_0 ].
\end{gather*}
Then
\[ \frac{2\pi}{q_0} I_{\chi_1, \chi_2}( w^- ) \leq \int_{T/2}^T \left| L_{\chi_1, \chi_2}\left( \tfrac12 + \i t \right) \right|^2 \, \d t \leq \frac{2\pi}{q_0} I_{\chi_1, \chi_2}( w^+ ), \]
so that after applying \cref{501} on both sides, we arrive at the following asymptotic formula,
\[ \int_{T/2}^T \left| L_{\chi_1, \chi_2}\left( \tfrac12 + \i t \right) \right|^2 \, \d t = \int_{T/2}^T \! P_{\chi_1, \chi_2}(\log t) \, \d t + \LO{ \frac{ E_{\chi_1, \chi_2}( q_0 T, \Omega) }{ q_0 } + \Omega T }. \]
Now Theorems~\ref{101}, \ref{103} and~\ref{105} follow with the choice
\[ \Omega = (q_1, q_2)^{-\frac13} \left( q_1^\ast q_1 + q_2^\ast q_2 \right)^\frac13 (q_1 q_2)^\frac13 T^{-\frac13}. \]

\subsection{An approximative formula for~\texorpdfstring{\( | L_{\chi_1, \chi_2}(s) |^2 \)}{|L(s)|\^{}2}} \label{51}

As a first step towards the proof of \cref{501}, we will develop here an approximative formula for~\( | L_{\chi_1, \chi_2}(s) |^2 \) on the critical line.

In order to state the exact result, we first choose a smooth weight function~\( V : (0, \infty) \to [0, 1] \) which satisfies the conditions
\begin{equation} \label{5179}
  V(\xi) + V( \xi^{-1} ) = 1 \quad \text{for} \quad \xi > 0 \qquad \text{and} \qquad V(\xi) = 0 \quad \text{for} \quad \xi \geq 2.
\end{equation}

Then the formula reads as follows.

\begin{proposition} \label{511}
  Let~\( \delta, \eps > 0 \) and~\( \rho > 1 \).
  Then we have, for~\( t^{1 - \delta} \gg q_0 \max\{ q_1, q_2 \} \),
  \begin{equation} \label{5197}
    \left| L_{\chi_1, \chi_2}\left( \tfrac12 + \i \tfrac{2\pi t}{q_0} \right) \right|^2 = 2 \Re\big( \Sigma_{\chi_1, \chi_2}^{ (1) }(t) + \Sigma_{\chi_1, \chi_2}^{ (2) }(t) \big) + R_{\chi_1, \chi_2}(t),
  \end{equation}
  where \( \Sigma_{\chi_1, \chi_2}^{ (1) }(t) \)~and~\( \Sigma_{\chi_1, \chi_2}^{ (2) }(t) \) are given by
  \begin{align*}
    \Sigma_{\chi_1, \chi_2}^{ (1) }(t) &:= \sum_{n_1, n_2 = 1}^\infty \frac{ \tau_{\chi_1, \chi_2}(n_1) \tau_{ \overline{\chi_1}, \overline{\chi_2} }(n_2) }{ (n_1 n_2)^\frac12 } e\left( \frac t{q_0} \log\left( \frac{n_2}{n_1} \right) \right) W_{1, \rho}\left( \frac{n_1}t, \frac{ n_2 }t \right), \\
    \Sigma_{\chi_1, \chi_2}^{ (2) }(t) &:= \alpha_{\chi_1, \chi_2}\left( \tfrac12 + \i \tfrac{2\pi t}{q_0} \right) \sum_{n_1, n_2 = 1}^\infty \frac{ \tau_{ \overline{\chi_1}, \overline{\chi_2} }(n_2) \tau_{ \overline{\chi_1}, \overline{\chi_2} }(n_1) }{ (n_1 n_2)^\frac12 } e\left( \frac t{q_0} \log(n_1 n_2) \right) W_{2, \rho}\left( \frac{ n_1 }t, \frac{ n_2 }t \right),
  \end{align*}
  with
  \[ W_{1, \rho}(\xi_1, \xi_2) := V(\xi_1) \left( 1 - V\left( {\xi_2}^{-1} \right) V\left( \rho {\xi_2}^{-1} \right) \right) \qquad \text{and} \qquad W_{2, \rho}(\xi_1, \xi_2) := V(\xi_1) V(\xi_2) V( \rho \xi_2 ), \]
  and where \( R_{\chi_1, \chi_2}(t) \) is bounded by
  \[ R_{\chi_1, \chi_2}(t) \ll q_0 t^{ -\frac14 + \eps } \qquad \text{and} \qquad \int_{T_0 / 2}^{T_0} \left| R_{\chi_1, \chi_2}(t) \right| \, \d t \ll q_0 {T_0}^{\frac 38 + \eps} \quad \text{for} \quad {T_0}^{1 - \delta} \gg q_0 \max\{ q_1, q_2 \}. \]
  The implicit constants depend at most on~\(V\), \(\rho\), \(\delta\) and~\( \eps \).
\end{proposition}
\begin{proof}
  \cref{511} is essentially a direct consequence of the approximate functional equations stated in \cref{24}.
  
  We open the square and apply \cref{241} twice with~\( \sigma = 1/2 \) and~\( x = y = t  \).
  After taking account of~\eqref{2237}, this gives
  \[ \left| L_{\chi_1, \chi_2}\left( \tfrac12 + \i \tfrac{2\pi t}{q_0} \right) \right|^2 = 2\Re( \Sigma_1(t) + \Sigma_2(t) + R_1(t) + R_2(t) ) + R_3(t), \]
  with
  \begin{align*}
    \Sigma_1(t) &:= \sum_{n_1, n_2 = 1}^\infty \frac{ \tau_{\chi_1, \chi_2}(n_1) \tau_{ \overline{\chi_1}, \overline{\chi_2} }(n_2) }{ (n_1 n_2)^\frac12 } e\left( \frac t{q_0} \log\left( \frac{n_2}{n_1} \right) \right) V\left( \frac{ n_1 }t \right) V\left( \frac{n_2}t \right), \\
    \Sigma_2(t) &:= \overline{ \alpha_{\chi_1, \chi_2}\left( \tfrac12 + \i \tfrac{2\pi t}{q_0} \right) } \sum_{n_1, n_2 = 1}^\infty \frac{ \tau_{ \chi_1, \chi_2 }(n_1) \tau_{ \chi_1, \chi_2 }(n_2) }{ (n_1 n_2)^\frac12 } e\left( -\frac t{q_0} \log(n_1 n_2) \right) V\left( \frac{ n_1 }t \right) V\left( \frac{ n_2}t \right), \\
    \intertext{and}
    R_1(t) &:= \overline{ \alpha_{\chi_1, \chi_2}\left( \tfrac12 + \i \tfrac{2\pi t}{q_0} \right) } R_{\chi_1, \chi_2}\left( \tfrac12 + \i \tfrac{2\pi t}{q_0}; t, t \right) \sum_{n = 1}^\infty \frac{ \tau_{ \chi_1, \chi_2 }(n) }{ n^\frac12 } e\left( -\frac t{q_0} \log n \right) V\left( \frac nt \right), \\
    R_2(t) &:= \overline{ R_{\chi_1, \chi_2}\left( \tfrac12 + \i \tfrac{2\pi t}{q_0}; t, t \right) } \sum_{n = 1}^\infty \frac{ \tau_{ \chi_1, \chi_2 }(n) }{ n^\frac12 } e\left( -\frac t{q_0} \log n \right) V\left( \frac nt \right), \\
    R_3(t) &:= \left| R_{\chi_1, \chi_2}\left( \tfrac12 + \i \tfrac{2\pi t}{q_0}; t, t \right) \right|^2.
  \end{align*}
  
  Next, we use \cref{242} with~\( \sigma = 1/2 \) and~\( x = y = t \) to express~\( \Sigma_2(t) \) as
  \[ \Sigma_2(t) = \Sigma_2'(t) + \Sigma_2''(t) + R_4(t), \]
  with
  \begin{align*}
    \Sigma_2'(t) &:= \sum_{n_1, n_2 = 1}^\infty \frac{ \tau_{\chi_1, \chi_2}(n_1) \tau_{ \overline{\chi_1}, \overline{\chi_2} }(n_2) }{ (n_1 n_2)^\frac12 } e\left( \frac t{q_0} \log\left( \frac{n_2}{n_1} \right) \right) V\left( \frac{ n_1 }t \right) V\left( \frac t{n_2} \right) V\left( \frac{n_2}{\rho t} \right), \\
    \Sigma_2''(t) &:= \overline{ \alpha_{\chi_1, \chi_2}\left( \tfrac12 + \i \tfrac{2\pi t}{q_0} \right) } \sum_{n_1, n_2 = 1}^\infty \frac{ \tau_{ \chi_1, \chi_2 }(n_1) \tau_{ \chi_1, \chi_2 }(n_2) }{ (n_1 n_2)^\frac12 } e\left( \frac{-t}{q_0} \log(n_1 n_2) \right) V\left( \frac{ n_1 }t \right) V\left( \frac{ n_2 }t \right) V\left( \frac{ \rho n_2 }t \right), \\
    R_4(t) &:= \overline{ R_{\chi_1, \chi_2}'\left( \tfrac12 + \i \tfrac{2\pi t}{q_0}; t, t \right) } \sum_{n = 1}^\infty \frac{ \tau_{ \chi_1, \chi_2 }(n) }{ n^\frac12 } e\left( -\frac t{q_0} \log n \right) V\left( \frac nt \right).
  \end{align*}
  
  The terms \( \Sigma_1(t) \), \( \Sigma_2'(t) \) and~\( \Sigma_2''(t) \) together form the two main terms in~\eqref{5197}.
  Furthermore, it follows immediately from the bounds~\eqref{2440} and~\eqref{2416} that
  \[ R_3(t) \ll {q_0}^2 t^{ -\frac54 + \eps } \qquad \text{and} \qquad \int_{ T_0 / 2}^{T_0} \! \left| R_3(t) \right| \, \d t \ll {q_0}^2 {T_0}^{ -\frac58 + \eps } \quad \text{for} \quad {T_0}^{1 - \delta} \gg q_0 \max\{ q_1, q_2 \}. \]
  In order to estimate the other error terms we first note that
  \[ \sum_{n = 1}^\infty \frac{ \tau_{ \chi_1, \chi_2 }(n) }{ n^\frac12 } e\left( - \frac t{q_0} \log n \right) V\left( \frac nt \right) \ll t^{\frac 38 + \eps}, \]
  as can be shown by a standard counter integration argument using \cref{221}.
  Together with the bound~\eqref{2440}, we thus get, for~\( i = 1, 2, 4 \),
  \[ R_i(t) \ll q_0 t^{ -\frac14 + \eps } \qquad \text{and} \qquad \int_{ T_0 / 2}^{T_0} \left| R_i(t) \right| \, \d t \ll q_0 {T_0}^{\frac 38 + \eps} \quad \text{for} \quad {T_0}^{1 - \delta} \gg q_0 \max\{ q_1, q_2 \}. \]
  This finishes the proof of \cref{511}.
\end{proof}

\subsection{A preliminary formula for~\texorpdfstring{\( I_{\chi_1, \chi_2}(w) \)}{I(w)}} \label{52}

Next, we will use \cref{511} to prove a preliminary formula for~\( I_{\chi_1, \chi_2}(w) \) which reduces its estimation to the estimation of certain divisor sums.

Before stating the result, it is again necessairy to fix a smooth weight functions of a certain shape.
Let~\( U : \RR \to [0, \infty) \) be a smooth and compactly supported function such that
\[ U(\xi) = 1 \quad \text{for} \quad |\xi| \leq q_0 \Omega^{-1} {T_0}^{-7/8} \qquad \text{and} \qquad U(\xi) = 0 \quad \text{for} \quad |\xi| \geq 2 q_0 \Omega^{-1} {T_0}^{-7/8}, \]
and such that its derivatives satisfy
\begin{equation} \label{5210}
  U^{ (\nu) }(\xi) \ll |\xi|^{-\nu} \quad \text{for} \quad \nu \geq 0.
\end{equation}
Then we have the following formula for~\( I_{\chi_1, \chi_2}(w) \).

\begin{proposition} \label{521}
  Let~\( \delta, \eps > 0 \).
  Then we have
  \[ I_{\chi_1, \chi_2}(w) = 2 \Re\left( M_{\chi_1, \chi_2}^{ (1) }(w) + M_{\chi_1, \chi_2}^{ (2) }(w) \right) + \LO{ q_0 {T_0}^{\frac 38 + \eps} }, \]
  where
  \begin{align}
    M_{\chi_1, \chi_2}^{ (1) }(w) &:= \int \! \sum_{n = 1}^\infty \frac{ \left| \tau_{\chi_1, \chi_2}(n) \right|^2 }n V\left( \frac nt \right) w(t) \, \d t, \label{5205} \\
    M_{\chi_1, \chi_2}^{ (2) }(w) &:= \int \! \sum_\twoln{ n_1, n_2 \geq 1 }{ n_1 \neq n_2 } \frac{ \tau_{\chi_1, \chi_2}(n_1) \tau_{ \overline{\chi_1}, \overline{\chi_2} }(n_2) }{ (n_1 n_2)^\frac12 } e\left( \frac t{q_0} \log\left( \frac{n_2}{n_1} \right) \right) U\left( \frac{n_2}{n_1} - 1 \right) V\left( \frac{ n_1 }t \right) w(t) \, \d t. \label{5222}
  \end{align}
  The implicit constant depends at most on~\(V\), \(\delta\), \(\eps\) and the implicit constants in~\eqref{5005} and~\eqref{5210}.
\end{proposition}
\begin{proof}
  We apply \cref{511} with~\( \rho = 8 \) on the integrand in~\( I_{\chi_1, \chi_2}(w) \) and then integrate over~\(t\).
  This leads to
  \[ I_{\chi_1, \chi_2}(w) = 2\Re\left( J_1 + J_2 \right) + \LO{ q_0 {T_0}^{\frac 38 + \eps} }, \]
  with
  \begin{align*}
    J_1 &:= \sum_{n_1, n_2 = 1}^\infty \frac{ \tau_{\chi_1, \chi_2}(n_1) \tau_{ \overline{\chi_1}, \overline{\chi_2} }(n_2) }{ (n_1 n_2)^\frac12 } \int \! W_{1, 8}\left( \frac{n_1}t, \frac{ n_2 }t \right) e\left( \frac t{q_0} \log\left( \frac{n_2}{n_1} \right) \right) w(t) \, \d t, \\
    J_2 &:= \sum_{n_1, n_2 = 1}^\infty \frac{ \tau_{ \overline{\chi_1}, \overline{\chi_2} }(n_2) \tau_{ \overline{\chi_1}, \overline{\chi_2} }(n_1) }{ (n_1 n_2)^\frac12 } \int \! \alpha_{\chi_1, \chi_2}\left( \tfrac12 + \i \tfrac{2\pi t}{q_0} \right) W_{2, 8}\left( \frac{ n_1 }t, \frac{ n_2 }t \right) e\left( \frac t{q_0} \log(n_1 n_2) \right) w(t) \, \d t.
  \end{align*}
  We split the sum~\( J_1 \) into three parts as follows,
  \[ J_1 = \sum_\twoln{n_1, n_2 \geq 1}{ n_1 = n_2 } (\ldots) + \sum_\twoln{n_1, n_2 \geq 1}{ n_1 \neq n_2 } U\left( \frac{n_2}{n_1} - 1 \right) (\ldots) + \sum_\twoln{n_1, n_2 \geq 1}{ n_1 \neq n_2 } \left( 1 - U\left( \frac{n_2}{n_1} - 1 \right) \right) (\ldots) =: J_{1 \text a} + J_{1 \text b} + J_{1 \text c}. \]
  As we will see, the contribution coming from the sums \( J_{1 \text c} \)~and~\( J_2 \) is neglible, while both \( J_{1 \text a} \)~and~\( J_{1 \text b} \) contribute to the main term.
  
  We start with~\( J_{1 \text c} \).
  In this sum we have, by definition of~\(U\),
  \[ \left| \log\left( \frac{n_2}{n_1} \right) \right| \gg \min\left\{ 1, \left| \frac{n_2}{n_1} - 1 \right| \right\} \gg \min\left\{ 1, \frac{q_0}{ \Omega {T_0}^\frac78 } \right\}, \]
  and by integrating by parts over~\(t\) repeatedly, we see that the integral in~\( J_{1 \text c} \) gets arbitrarily small.
  Hence the contribution of~\( J_{1 \text c} \) is indeed negligible.
  
  Next, we consider~\( J_2 \).
  Using the approximation~\eqref{2263}, we can write the integral in~\(J_2\) as
  \[ \int \! \alpha_{\chi_1, \chi_2}\left( \tfrac12 + \i \tfrac{2\pi t}{q_0} \right) W_{2, 8}\left( \frac{ n_1 }t, \frac{ n_2 }t \right) e\left( \frac t{q_0} \log(n_1 n_2) \right)  w(t) \, \d t = \int \! e( F_1(t) ) F_2(t) \, \d t, \]
  with
  \[ F_1(t) := \frac t{q_0} \log\left( \frac{ e^2 n_1 n_2 }{t^2} \right) \qquad \text{and} \qquad F_2(t) := \i \frac{ G(\chi_1) G(\chi_2) }{ (-1)^{ \kappa_1 + \kappa_2 } q_0 } A\left( \frac{2\pi t}{q_0} \right) W_{2, 8}\left( \frac{ n_1 }t, \frac{ n_2 }t \right) w(t). \]
  The function~\( W_{2, 8}(t) \) vanishes unless both the conditions
  \[ \frac t{ n_1 } \geq \frac12 \quad \text{and} \quad \frac t{ n_2 } \geq 4, \]
  are met, which means that
  \[ \frac{ t^2 }{ n_1 n_2 } \geq \max\bigg\{ \frac14 \frac{n_1}{n_2}, 16 \frac{n_2}{n_1} \bigg\} \geq 2. \]
  This leads to the following lower bound for~\( F_1'(t) \),
  \[ F_1'(t) = \frac1{q_0} \log\left( \frac{ n_1 n_2 }{ t^2 } \right) \gg \frac1{q_0}, \]
  and integrating by parts repeatedly shows that the integral gets arbitrarily small.
  We thus see that the contribution of~\( J_2 \) too is neglible.
  
  Finally we turn towards the two remaining terms~\( J_{1 \text a} \) and~\( J_{1 \text b} \).
  In both these terms, it is certainly true that~\( 2 n_1 \geq n_2 \), at least for \(T_0\) sufficiently large.
  Since the integrand vanishes unless~\( n_1 \leq 2t \), this implies that~\( n_2 \leq 4t \).
  By consequence, the weight function~\( W_{1, 8} \) simplifies to
  \[ W_{1, 8}\left( \frac{ n_1 }t, \frac{ n_2 }t \right) = V\left( \frac{ n_1 }t \right). \]
  This finishes the proof of \cref{521}.
\end{proof}

In order to prove \cref{501}, it thus remains to evaluate the two sums inside \eqref{5205}~and~\eqref{5222}.
The evaluation of the former is fairly easy and will be done in \cref{53}, where we will prove the following asymptotic formula.

\begin{proposition} \label{522}
  Let~\( \eps > 0 \).
  Then
  \begin{equation} \label{5223}
    M_{\chi_1, \chi_2}^{ (1) }(w) = \int \! P_{\chi_1, \chi_2}^{ (1) }(\log t) w(t) \, \d t + \LO{ {q_0}^\frac12 {T_0}^{\frac12 + \eps} },
  \end{equation}
  where \( P_{\chi_1, \chi_2}^{ (1) } \) is a polynomial of degree less or equal to~\(4\) whose coefficients depend only on~\( \chi_1 \), \( \chi_2 \) and~\( V \).
  The implicit constant depends at most on~\(V\), \(\eps\) and the implicit constants in~\eqref{5005}.
\end{proposition}

The evaluation of the other sum is far more difficult, and it is here that the shifted convolution problem considered in \cref{4} comes up.
The final result, proven in \cref{54}, is as follows.

\begin{proposition} \label{523}
  Let~\( \eps > 0 \).
  Then
  \begin{equation} \label{5216}
    M_{\chi_1, \chi_2}^{ (2) }(w) = \int \! P_{\chi_1, \chi_2}^{ (2) }(\log t) w(t) \, \d t + \LO{ {T_0}^\eps E_{\chi_1, \chi_2}( T_0, \Omega) },
  \end{equation}
  where \( P_{\chi_1, \chi_2}^{ (2) } \) is a polynomial of degree less or equal to~\(2\) whose coefficients depend only on~\( \chi_1 \), \( \chi_2 \) and~\( V \), and where \( E_{\chi_1, \chi_2}( T_0, \Omega) \) is the quantity defined in~\eqref{5057}.
  The implicit constant depends at most on~\(V\), \(\eps\) and the implicit constants in~\eqref{5005} and~\eqref{5210}.
\end{proposition}

These two results, applied on the preliminary asymptotic estimate stated in \cref{521}, eventually give \cref{501}.
The polynomials appearing in \eqref{5223}~and~\eqref{5216} both depend on the specific choice of the weight function~\(V\).
However, as one would expect, all the terms containing~\(V\) cancel out at the end, and the polynomial~\( P_{\chi_1, \chi_2} \) appearing in the main term in \cref{501} is of course independent of~\(V\).
We will show this also explicitly in \cref{55}, where we will evaluate~\( P_{\chi_1, \chi_2} \) and express it as a residue.

\subsection{Evaluation of~\texorpdfstring{\( M_{\chi_1, \chi_2}^{ (1) }(w) \)}{M\^{}(1)(w)}} \label{53}

In order to prove \cref{522}, we only need to evaluate the sum over~\(n\) in~\eqref{5205}, which we can do by a standard contour integration argument.

An elementary calculation shows that, for~\( \Re(z) > 0 \), 
\[ T_{\chi_1, \chi_2}(z) := \sum_{n = 1}^\infty \frac{ \left| \tau_{\chi_1, \chi_2}(n) \right|^2 }{ n^{1 + z} } = \frac{ \psi_z(q_1) \psi_z(q_2) \zeta(1 + z)^2 L( 1 + z, \overline{\chi_1} \chi_2 ) L( 1 + z, \chi_1 \overline{\chi_2} ) }{ \psi_{1 + 2z}(q_1 q_2) \zeta(2 + 2z) }, \]
with~\( \psi_z(q) \) as defined in~\eqref{4086}.
By Mellin inversion we thus have
\[ \sum_{n = 1}^\infty \frac{ \left| \tau_{\chi_1, \chi_2}(n) \right|^2 }n V\left( \frac nt \right) = \frac1{2\pi\i} \int_{ (2) } \! \hat V(z) T_{\chi_1, \chi_2}(z) t^z \, \d z. \]
After moving the line of integration to~\( \Re(z) = -1/2 + \eps \) and using the following bound, valid in the critical strip,
\[ T_{\chi_1, \chi_2}(z) \ll {q_0}^{ 1 - \Re(z) + \eps } ( 1 + | \Im z | )^{ \frac{ 1 - \Re(z) }2 + \eps}, \]
we get
\[ \sum_{n = 1}^\infty \frac{ \left| \tau_{\chi_1, \chi_2}(n) \right|^2 }n V\left( \frac nt \right) = P_{\chi_1, \chi_2}^{ (1) }(\log t) + \LO{ {q_0}^\frac12 t^{-\frac12 + \eps} }, \]
where \( P_{\chi_1, \chi_2}^{ (1) } \) is the polynomial defined by
\begin{equation} \label{5311}
  P_{\chi_1, \chi_2}^{ (1) }(\log t) := \Res{z = 0} \! \left( \hat V(z) T_{\chi_1, \chi_2}(z) t^z \right).
\end{equation}
This proves \cref{522}.

\subsection{Evaluation of~\texorpdfstring{\( M_{\chi_1, \chi_2}^{ (2) }(w) \)}{M\^{}(2)(w)}} \label{54}

We start by introducing a new variable~\( h := n_2 - n_1 \) and splitting the ranges of~\(h\) and~\( n_1 \) into dyadic intervals via the dyadic partition of unity defined in~\eqref{4195}.
This way \( M_{\chi_1, \chi_2}^{ (2) }(w) \) is split up into sums of the form
\begin{align*}
  D^\pm(N, H) &:= \sum_{ h, n } \tau_{\chi_1, \chi_2}(n) \tau_{ \overline{\chi_1}, \overline{\chi_2} }(n + h) \int \! f^\pm(n, h; t) e\left( \frac t{q_0} \log\left( 1 + \frac hn \right) \right) w(t) \, \d t,
  \intertext{with}
  f^\pm(\xi, \eta; t) &:= \xi^{-\frac12} (\xi + \eta)^{-\frac12} u\left( \frac\xi N \right) u\left( \pm\frac\eta H \right) U\left( \frac\eta\xi \right) V\left( \frac\xi t \right).
\end{align*}
Integrating by parts over~\(t\) repeatedly shows that \( D^\pm(N, H) \) becomes negligibly small unless
\[ H \ll \frac{ q_0 N }{ \Omega {T_0}^{1 - \eps} }. \]
Similarly, we can assume that~\( {T_0}^\frac12 \ll N \ll T_0 \), since otherwise \( D^\pm(N, H) \) is either empty or can be included in the error term in~\eqref{5216}.

Next, we write the oscillating factor in the integral over~\(t\) as
\[ e\left( \frac t{q_0} \log\left( 1 + \frac hn \right) \right) = e\left( \frac{th}{q_0 n} \right) g\left( \frac t{q_0}, \frac hn \right) + \LO{ {T_0}^{-\frac53 + \eps} }, \]
with
\[ g(\zeta_1, \zeta_2) := \sum_{\ell = 0}^{10} \frac{ (-2\pi \i \zeta_1)^\ell }{\ell !} \left( \sum_{k = 0}^5 \frac{ (-\zeta_2)^{k + 2} }{ k + 2 } \right)^\ell, \]
and then integrate by parts over~\(t\), so that
\[ D^\pm(N, H) = \int \! D_{1, t}^\pm(N, H) w'(t) \, \d t + \int \! D_{2, t}^\pm(N, H) \frac{ w(t) }t \, \d t + \LO{ 1 }, \]
where \( D_{i, t}^\pm(N, H) \) is given by
\[ D_{i, t}^\pm(N, H) := \sum_h \frac1h \sum_n \tau_{\chi_1, \chi_2}(n) \tau_{ \overline{\chi_1}, \overline{\chi_2} }(n + h) f_{i, t}^\pm( n, h ) e\left( \frac{th}{q_0 n} \right), \]
with
\[ f_{1, t}^\pm(\xi, \eta) := - \frac{q_0 \xi}{2\pi\i} f^\pm(\xi, \eta; t) g\left( \frac t{q_0}, \frac\xi\eta \right) \qquad \text{and} \qquad f_{2, t}^\pm(\xi, \eta) := t \frac\partial{\partial t} f_{1, t}^\pm(\xi, \eta). \]

Here we use \cref{401} with~\( \alpha = t / q_0 \) to evaluate the two sums~\( D_{1, t}^\pm(N, H) \) and~\( D_{2, t}^\pm(N, H) \).
After reversing the integration by parts in the appearing main term, we get
\[ D^\pm(N, H) = \int \! M_t^\pm(N, H) w(t) \, \d t + \LO{ {T_0}^\eps E_{\chi_1, \chi_2}( T_0, \Omega) }, \]
where \( E_{\chi_1, \chi_2}( T_0, \Omega) \) is as defined in~\eqref{5057}, and where
\[ M_t^\pm(N, H) := \sum_h \frac1h \int \! Q_{\chi_1, \chi_2}( \log\xi, \log(\xi + h); h ) f^\pm(\xi, h; t) g\left( \frac t{q_0}, \frac\xi\eta \right) e\left( \frac{th}{q_0 \xi} \right) \, \d\xi. \]
Integration by parts over~\(\xi\) shows that~\( M_t^\pm(N, H) \) becomes negligibly small if~\( H \gg q_0 N^{1 + \eps} {T_0}^{-1} \), while for~\( H \ll q_0 N^{1 + \eps} {T_0}^{-1} \) it simplifies to
\[ M_t^\pm(N, H) = \frac{q_0}{2\i} \sum_h \frac{ u\left( \frac{\pm  h}H \right) }{\pi h} \! \int \! \frac\partial{\partial\xi} \left( \! Q_{\chi_1, \chi_2}( \log(t\xi), \log(t\xi); h ) \xi u\left( \frac{t \xi}N \right) V(\xi) \! \right) e\left( \frac h{q_0 \xi} \right) \, \d\xi + \LO{ \frac{ {T_0}^\eps }{ {T_0}^\frac12 } }. \]
Finally, we sum over all~\( H \ll q_0 N^{1 + \eps} {T_0}^{-1} \) and~\( {T_0}^\frac12 \ll N \ll T_0 \), and then complete the sum over~\(h\) and the integral over~\(\xi\) trivially.
This gives
\[ M_{\chi_1, \chi_2}^{ (2) }(w) = \int \! P_{\chi_1, \chi_2}^{ (2) }(\log t) w(t) \, \d t + \LO{ {T_0}^\eps E_{\chi_1, \chi_2}(T_0, \Omega) }, \]
with
\begin{equation} \label{5446}
  P_{\chi_1, \chi_2}^{ (2) }(\log t) := \frac{q_0}{2\i} \sum_{ h \in \ZZ \setminus \{0\} } \frac1{\pi h} \int \! \frac\partial{\partial\xi} \left( Q_{\chi_1, \chi_2}( \log(t \xi), \log(t \xi); h ) \xi V(\xi) \right) e\left( \frac h{q_0 \xi} \right) \, \d\xi,
\end{equation}
which is what we wanted to show.

\subsection{The main term} \label{55}

Here we want to evaluate the polynomial~\( P_{\chi_1, \chi_2} \) which appears in \cref{501} and which is given by
\[ P_{\chi_1, \chi_2}(\log t) = 2 \Re\left( P_{\chi_1, \chi_2}^{ (1) }\left( \log \frac{q_0 t}{2\pi} \right) + P_{\chi_1, \chi_2}^{ (2) }\left( \log \frac{q_0 t}{2\pi} \right) \right), \]
where \( P_{\chi_1, \chi_2}^{ (1) } \)~and~\( P_{\chi_1, \chi_2}^{ (2) } \) are the polynomials coming up in Propositions~\ref{522} and~\ref{523}.
Our treatment follows closely the path set out by Conrey~\cite{Con96}.

We will focus on the case~\( \chi_1 = \chi_2 \).
Since the Laurent series expansion of~\( \hat V(z) \) around~\( z = 0 \) is given by
\[ \hat V(z) = \frac1z - \sum_{\ell = 0}^\infty \frac{ z^{2\ell + 1} }{ (2\ell + 2)! } \int_0^\infty \! V'(\xi) (\log\xi)^{2\ell + 2} \, \d\xi, \]
we immediately see by~\eqref{5311} that
\begin{multline*}
  P_{\chi_1, \chi_1}^{ (1) }(\log t) = \Res{z = 0}\!\left( \frac{ Z_{q_1}(z)^4 }{ \psi_{1 + 2z}(q_1) \zeta(2 + 2z) } \frac{ t^z }{ z^5 } \right) \\
    - \frac12 \, \Res{z = 0}\!\left( \frac{ Z_{q_1}(z)^4 }{ \psi_{1 + 2z}(q_1) \zeta(2 + 2z) } \frac{ t^z }{z^3} \right) \int_0^\infty \! V'(\xi) (\log\xi)^2 \, \d\xi - \frac1{24} \frac{ \psi_0(q_1)^4 }{ \psi_1(q_1) \zeta(2) } \int \! V'(\xi) (\log\xi)^4 \, \d\xi,
\end{multline*}
with \( Z_q(z) \) and~\( \psi_z(q) \) as defined in~\eqref{4086}.

The evaluation of the other polynomial~\( P_{\chi_1, \chi_1}^{ (2) } \) proves more difficult.
By~\eqref{5446} we can write it as
\[ P_{\chi_1, \chi_1}^{ (2) }(\log t) = \Delta_{z_1} \Delta_{z_2} \psi_0(q_1) Z_{q_1}(2 z_1) Z_{q_1}(2 z_2) Z_{q_1}(2 z_1 + 2 z_2) t^{z_1 + z_2} \frac{ A(z_1 + z_2) }{ B(z_1 + z_2) }, \]
with
\begin{align*}
  A(z) &:= \psi_{1 + 2z}(q_1) \zeta(2 + 2z) \sum_{h = 1}^\infty \frac{ r_{q_1}(h) }{\pi h} \sum_\twoln{ c = 1 }{ (c, q_1) = 1 }^\infty \frac{ r_c(h) }{ c^{2 + 2z} } \int_0^\infty \! \frac\partial{\partial \xi} \left( V(\xi) \xi^{1 + z} \right) \sin\left( 2\pi \frac h{ \xi q_1 } \right) \, \d\xi, \\
  B(z) &:= \psi_{1 + 2z}(q_1) \zeta(2 + 2z) \psi_0(q_1) Z_{q_1}(2z).
\end{align*}
Note that the expression~\( A(z) \) converges in a neighbourhood of~\( z = 0 \), and thus defines a holomorphic function in this region.
A simple calculation then shows that
\[ P_{\chi_1, \chi_1}^{ (2) }(\log t) = \frac{\partial^2}{\partial z^2} \left( Z_{q_1}(z)^4 t^z \frac{ A(z) }{ B(z) } \right) \bigg|_{z = 0} = 2 \, \Res{z = 0}\!\left( Z_{q_1}(z)^4 \frac{ A(z) }{ B(z) } \frac{ t^z }{z^3} \right). \]
In order to evaluate~\( P_{\chi_1, \chi_1}^{ (2) } \), we therefore need to determine the first three terms in the Taylor expansion of~\( A(z) \) around~\( z = 0 \).

In order to avoid unnecessairy convergence issues, we will assume in the following transformations that~\( z > 0 \) .
Using
\[ r_{q_1}(h) = \sum_{ d \mid (q_1, h) } \mu\left( \frac{q_1}d \right) d \qquad \text{and} \qquad \sum_\twoln{c = 1}{ (c, q_1) = 1}^\infty \frac{ r_c(h) }{ c^{2 + 2z} } = \frac1{ \psi_{1 + 2z}(q_1) \zeta(2 + 2z) } \sum_\twoln{h_1 \mid h}{ (h_1, q_1) = 1 } \frac1{ {h_1}^{1 + 2z} }, \]
we can write~\( A(z) \) as
\[ A(z) = \sum_\twoln{h_1 = 1}{ (h_1, q_1) = 1 }^\infty \sum_{d \mid q_1} \frac{ \mu(d) }{ {h_1}^{2 + 2z} } \sum_{h_2 = 1}^\infty \frac1{\pi h_2} \int_0^\infty \! \frac\partial{\partial \xi} \left( V(\xi) \xi^{1 + z} \right) \sin\left( 2\pi \frac{h_1 h_2}{ d \xi } \right) \, \d\xi. \]
Since the sum over~\( h_2 \) is boundedly convergent (see~\cite[p.~4]{Iwa97}), we can exchange summation and integration.
By~\cite[(1.5)]{Iwa97} we then get
\[ A(z) = \sum_\twoln{h = 1}{ (h, q_1) = 1 }^\infty \frac1{ h^{1 + z} } \sum_{d \mid q_1} \frac{ \mu(d) }{ d^{1 + z} } \int_0^\infty \! \frac\partial{\partial \xi} \left( V\left( \frac h{ d \xi } \right) \frac1{ \xi^{1 + z} } \right) \left( \xi - [\xi] - \frac12 \right) \, \d\xi, \]
where \( [\xi] \) denotes the integer part of~\(\xi\).
The integral over~\( \xi \) can now be evaluated via the Euler-Maclaurin summation formula, which gives
\begin{align*}
  A(z) &= \sum_\twoln{h = 1}{ (h, q_1) = 1 }^\infty \frac1{ h^{1 + z} } \sum_{d \mid q_1} \frac{ \mu(d) }{ d^{1 + z} } \left( \sum_{n = 1}^\infty V\left( \frac h{ d n } \right) \frac1{ n^{1 + z} } - \int_0^\infty \! V\left( \frac h{ d \xi } \right) \frac1{ \xi^{1 + z} } \, \d\xi \right) \\
    &= \sum_\twoln{h, n = 1}{ (hn, q_1) = 1 }^\infty \frac1{ (hn)^{1 + z} } V\left( \frac hn \right) - \frac{ \psi_0(q_1) Z_{q_1}(2z) }{2z} \hat V(z).
\end{align*}
By~\eqref{5179} the double sum on the last line becomes
\[ \sum_\twoln{h, n = 1}{ (hn, q_1) = 1 }^\infty \frac1{ (hn)^{1 + z} } V\left( \frac hn \right) = \frac12 \sum_\twoln{h, n = 1}{ (hn, q_1) = 1 }^\infty \frac1{ (hn)^{1 + z} } \left( V\left( \frac hn \right) + V\left( \frac nh \right) \right) = \frac{ Z_{q_1}(z)^2 }{ 2 z^2 }. \]
Hence
\[ A(z) = \frac{ Z_{q_1}(z)^2 }{ 2 z^2 } - \frac{ \psi_0(q_1) Z_{q_1}(2z) }{2z} \hat V(z), \]
which eventually leads to the following expression for~\( P_{\chi_1, \chi_1}^{ (2) } \),
\begin{multline*}
  P_{\chi_1, \chi_1}^{ (2) }(\log t) = \Res{z = 0} \! \left( \frac{ Z_{q_1}(z)^6 }{ \psi_0(q_1) \psi_{1 + 2z}(q_1) Z_{q_1}(2z) \zeta(2 + 2z) } \frac{ t^z }{z^5} - \frac{ Z_{q_1}(z)^4 }{ \psi_{1 + 2z}(q_1) \zeta(2 + 2z) } \frac{ t^z }{z^5} \right) \\
    + \frac12 \, \Res{z = 0} \! \left( \frac{ Z_{q_1}(z)^4 }{ \psi_{1 + 2z}(q_1) \zeta(2 + 2z) } \frac{ t^z }{z^3} \right) \int_0^\infty \! V'(\xi) (\log \xi)^2 \, \d\xi + \frac1{ 24 } \frac{ \psi_0(q_1)^4 }{ \psi_1(q_1) \zeta(2) } \int_0^\infty \! V'(\xi) (\log\xi)^4 \, \d\xi.
\end{multline*}

All in all, we end up with
\begin{align}
  P_{\chi_1, \chi_1}(\log t) &= \, \Res{z = 0}\!\left( \frac{ {q_1}^z \psi_z(q_1)^6 }{ \psi_0(q_1) \psi_{2z}(q_1) \psi_{1 + 2z}(q_1) } \frac{ \zeta(1 + z)^6 }{ (2\pi)^z \zeta(1 + 2z) \zeta(2 + 2z) } t^z \right). \label{5508} \\
  \intertext{Remember that \( \psi_z(q) \)~was defined in~\eqref{4086}.
    The cases where~\( \chi_1 \neq \chi_2 \) can be evaluated in the same manner.
    If~\( \chi_1 \neq \chi_2\) but~\( q_1 = q_2 \), we get}
  \begin{split}
    P_{\chi_1, \chi_2}(\log t) &= \Res{z = 0} \! \left( \frac{ {q_1}^z \psi_z(q_1)^4 }{ \psi_0(q_1) \psi_{1 + 2z}(q_1) \psi_{2z}(q_1) } \frac{ \zeta(1 + z)^4 L( 1 + z, \overline{\chi_1} \chi_2 ) L( 1 + z, \chi_1\overline{\chi_2} ) }{ (2\pi)^z \zeta(1 + 2z) \zeta(2 + 2z) } t^z \right) \\
      &\phantom{ = {} } + \Re\left( \frac{ \overline{ G(\chi_1) } G(\chi_2) }{ q_1 } \frac{ L( 1, \chi_1 \overline{\chi_2} )^4 }{ L( 2, ( \chi_1 \overline{\chi_2} )^2 ) } + \chi_1 \chi_2(-1) \frac{ G(\chi_1) \overline{ G(\chi_2) } }{ q_1 } \frac{ L( 1, \overline{\chi_1} \chi_2 )^4 }{ L( 2, ( \overline{\chi_1} \chi_2 )^2 ) } \right),
  \end{split} \label{5583}
  \intertext{while if~\( q_1 \neq q_2 \), we get}
  \begin{split}
    P_{\chi_1, \chi_2}(\log t) &= \, \Res{z = 0} \!\Bigg( \frac{ 2 (q_1 q_2)^z \psi_z(q_1)^2 \psi_z(q_2)^2 }{ \left( \psi_0(q_2) {q_1}^z \psi_{2z}(q_1) + \psi_0(q_1) {q_2}^z \psi_{2z}(q_2) \right) \psi_{1 + 2z}(q_1 q_2) } \\
      &\phantom{ = {} } \qquad\qquad\quad\qquad\qquad\qquad\quad \cdot \frac{ \zeta(1 + z)^4 L(1 + z, \overline{\chi_1} \chi_2 ) L(1 + z, \chi_1 \overline{\chi_2} ) }{ (2\pi)^z \zeta(1 + 2z) \zeta(2 + 2z) } t^z \Bigg).
  \end{split} \label{5576}
\end{align}
Note that the second term on the right hand side in~\eqref{5583} disappears if \(\chi_1\)~and~\(\chi_2\) do not have the same parity.

\vskip0.5cm

\noindent
Berke Topacogullari, EPFL SB MATH TAN, Station~8, 1015~Lausanne, Switzerland \\
\textit{E-mail:} \texttt{\href{mailto:berke.topacogullari@epfl.ch}{berke.topacogullari@epfl.ch}}

\end{document}